\documentclass{jcmlatex}
\usepackage{subfigure}
\usepackage{epstopdf}
\usepackage{algorithm}%
\usepackage{algorithmicx}%
\usepackage{algpseudocode}%
\usepackage{multirow}%
\usepackage{listings}%
\usepackage{comment}
\usepackage{float}
\usepackage{textcomp}%
\usepackage{manyfoot}%
\usepackage{booktabs}%
\usepackage{amsmath}
\usepackage{comment}
\usepackage{xcolor}
\usepackage{hyperref}
\hypersetup{
  hidelinks         
}
\def\JCMvol{xx}
\def\JCMno{x}
\def\JCMyear{200x}
\def\JCMreceived{xxx}
\def\JCMrevised{xxx}
\def\JCMaccepted{xxx}

\begin{document}

\markboth{Zilong Cui and Ran Gu}{Heuristic Bundle Upper Bound Based Polyhedral Bundle Method for semidefinite programming}

\title{Heuristic Bundle Upper Bound Based Polyhedral Bundle Method for Semidefinite Programming}

\author{Zilong Cui
\thanks{NITFID, School of Statistics and Data Science, Nankai University, Tianjin 300071, China\\ Email: zl\_c@mail.nankai.edu.cn}
\and
Ran Gu\footnotetext[1]{Corresponding author}\textsuperscript{1)}
\thanks{NITFID, School of Statistics and Data Science, Nankai University, Tianjin 300071, China\\ Email: rgu@nankai.edu.cn}}
\maketitle

\begin{abstract}
Semidefinite programming (SDP) is a fundamental class of convex optimization problems with diverse applications in mathematics, engineering, machine learning, and related disciplines. This paper investigates the application of the polyhedral bundle method to standard SDPs. The basic idea of this method is to approximate semidefinite constraints using linear constraints, and thereby transform the SDP problem into a series of quadratic programming subproblems. However, the number of linear constraints often increases continuously during the iteration process, leading to a significant reduction in the solution efficiency. Therefore, based on the idea of limiting the upper bound on the number of bundles, we heuristically select the upper bound through numerical experiments according to the rank of the primal optimal solution and propose a modified subproblem. In this way, under the premise of ensuring the approximation ability of the lower approximation model, we minimize the number of bundles as much as possible to improve the solution efficiency of the subproblems. The algorithm performs well in both Max-Cut problems and random SDP problems. In particular, for random sparse SDP problems with low condition numbers, under the condition of a relative accuracy of \(10^{-4}\), it shows a significant improvement compared with algorithms such as interior-point methods.
\end{abstract}



\section{Introduction}
\label{sec:into}
In this paper, we study the following primal SDP in the standard form
\begin{equation}\label{eq:Psdp}
\underset{X\in \mathbb{S}^{n}}{\min}\quad\langle C,X \rangle\quad \text{s.t.}\quad \mathcal{A}X=b  \quad \text{and}\quad X \succeq 0,
\end{equation}
where \(C \in \mathbb{S}^{n}\) and \(b \in \mathbb{R}^{m}\) where \(\mathbb{S}^{n}\) denotes the set of n-dimensional symmetric matrices. The linear map \(\mathcal{A}: \mathbb{S}^{n} \rightarrow\mathbb{R}^{m}\) has rank \(m\) which can be expressed explicitly as \(\mathcal{A}(X) = [\langle A_{1},X \rangle,\cdots, \langle A_{m},X \rangle]^{T}\) with each \( A_{i}\in \mathbb{S}^{n} \) for \( i = 1, 2, \cdots, m \). The matrices \( A_{i} (i = 1, 2, \cdots, m) \) are linearly independent in \( \mathbb{S}^{n} \). \(X \succeq 0\) denotes that the symmetric matrix \(X\) is positive semidefinite and it can also be expressed as \(X \in \mathbb{S}^{n}_{+}\), meanwhile the symbol \(\mathbb{S}^{n}_{++}\) denotes the set of n-dimensional positive definite matrices. Assuming that the primal variables have a unique solution we denote the rank of the optimal solution \(X^{\star}\) of the primal variables by \(r^{\star}\). The dual of the primal SDP (\ref{eq:Psdp}) is given by
\begin{equation}\label{eq:Dsdp}
\underset{y\in \mathbb{R}^{m}}{\min}-b^{T}y\quad\text{s.t.}\quad S=C-\mathcal{A}^{T}y\quad\text{and}\quad S\succeq 0,
\end{equation}
where the linear operator \(\mathcal{A}^{T}:\mathbb{R}^{m}\to \mathbb{S}^{n}\) denotes the adjoint operator of the linear operator \(\mathcal{A}\), and for any \(y \in \mathbb{R}^{m}\), \(\mathcal{A}^{T}y\) is given by \(\sum_{i = 1}^{m}y_{i}A_{i}\). For a positive semidefinite matrix \( S \in \mathbb{S}^{n}_{+} \), we define its condition number here as \( \mathcal{\bar{\kappa}}(S) = \frac{\lambda_{\max}(S)}{\lambda_{\min>0}(S)} \), where \( \lambda_{\min>0}(S) \) denotes the smallest non-zero eigenvalue of \( S \). In this paper, all the condition numbers we refer to correspond to the one defined here.

SDPs form a class of convex optimization problems and it can be regarded as an extension of linear programming (LP). The dimensionality of the matrix variable is an order of magnitude larger than that of the vector variable in LP, and the positive semidefinite constraint is non-polyhedral with a curved boundary, both of which make SDP difficult to solve. From the perspective of research history, the theoretical origin of SDP can be traced back to the paper published by Bellman and Fan in 1963 \cite{bellman1963systems}. However, problems related to SDP appeared earlier in the field of control theory and the starting point of research in this direction can be traced back to Lyapunov's classic research on the stability of dynamic systems in 1890, which laid an early foundation for the subsequent application of SDP in related fields. SDP has a wealth of classic application scenarios such as the positive semidefinite relaxation of the Max-Cut problem \cite{delorme1993laplacian}, matrix completion \cite{srebro2005rank}, phase retrieval \cite{chai2010array,waldspurger2015phase}, community detection \cite{mathieu2010correlation}, optimal power flow \cite{bai2008semidefinite}, combinatorial optimization \cite{alizadeh1995interior}, and the positive semidefinite penalty problem of quadratically constrained quadratic programming (QCQP) \cite{gu2021positive} and other fields.

\subsection{Assumption}
The optimality conditions (KKT conditions) for primal SDP (\ref{eq:Psdp}) and dual SDP (\ref{eq:Dsdp}) can be written as
\begin{equation}\label{kkt}
\mathcal{A}X-b=0,\quad\mathcal{A}^{T}y+S-C=0,\quad\langle X,S \rangle=0,\quad X\succeq0,\quad S\succeq0.
\end{equation}
However, the existence‌ of KKT point relies on the following assumptions.
\begin{assumption}\label{as:1}
For primal SDP (\ref{eq:Psdp}), there exists a feasible solution $X_{0} \in \mathbb{S}^{n}_{+}$ such that
\begin{equation*}
\mathcal{A}X_{0}-b=0,\quad X_{0}\succ0.
\end{equation*}
\end{assumption}
\begin{assumption}\label{as:2}
For dual SDP (\ref{eq:Dsdp}), there exists a feasible solution $(y_{0}, S_{0})\in \mathbb{R}^{m}\times \mathbb{S}^{n}_{+}$ such that
\begin{equation*}
\mathcal{A}^{T}y_{0}+S_{0}-C=0,\quad S_{0}\succ0.
\end{equation*}
\end{assumption}
Assumption \ref{as:1} and Assumption \ref{as:2} ensure that primal SDP (\ref{eq:Psdp}) and dual SDP (\ref{eq:Dsdp}) satisfy Slater’s constraint qualification. According to convex analysis\cite{borwein2006convex} and the above assumptions, the strong duality holds for primal SDP (\ref{eq:Psdp}) and dual SDP (\ref{eq:Dsdp}), and the KKT conditions (\ref{kkt}) have solutions \((X^{\star},y^{\star},S^{\star})\). From duality gap \(\langle X^{\star},S^{\star} \rangle=0\) and \(X^{\star}\succeq0,S^{\star}\succeq0\), we have the complementary slackness condition:
\begin{equation*}
X^{\star}S^{\star}=X^{\star}(C-\mathcal{A}^{T}y^{\star})=0.
\end{equation*}
The complementarity condition indicates that \( X^{\star}\) and \( S^{\star}\) commute (i.e., satisfying the condition that \( X^{\star}S^{\star} = S^{\star}X^{\star}\)), thus they share a common system of eigenvectors. Thus we have the following lemma \ref{lemma:1}\cite{alizadeh1997complementarity}.
\begin{lemma}\label{lemma:1}
Let \( X^{\star}\) and \((y^{\star},S^{\star})\) be respectively primal and dual feasible for (\ref{eq:Psdp}) and (\ref{eq:Dsdp}). Then they are optimal if and only if there exists \(Q\in \mathbb{R}^{n \times n}\), with \(Q^{T}Q = I\), such that
\begin{equation*}
X^{\star}=Q\text{diag}(\lambda_{1},\lambda_{2},\cdots,\lambda_{n})Q^{T},
S^{\star}=Q\text{diag}(w_{1},w_{2},\cdots,w_{n})Q^{T},
\end{equation*}
\text{and}
\begin{equation}\label{eq:cs}
\lambda_{i}w_{i}=0,i=1,\cdots,n
\end{equation}
\text{all hold}.
\end{lemma}
The equation (\ref{eq:cs}) implies that
\begin{equation*}
rank(X^{\star}) + rank(S^{\star}) \leq n.
\end{equation*}
In order to ensure that the semidefinite problem does not fall into a degenerate case, we need to make a stronger assumption, that is, every solution pair \((X^{\star},y^{\star},S^{\star})\) satisfies the following stronger strict complementarity condition:
\begin{equation*}
rank(X^{\star}) + rank(S^{\star}) = n.
\end{equation*}
This implies that for any \(i = 1, \dots, n\), exactly one of \(\lambda_i = 0\) or \(w_i = 0\) holds. It is worth noting that this assumption ensures that all solutions have the same rank. Therefore, in this case, the solution to the original problem is actually unique\cite{lemon2016low}. Moreover, this assumption holds for most practical problems, so it is a mild assumption.

\subsection{Related Works}
In this subsection, we introduce the algorithmic advancements in SDP. In the early stages, Nesterov and Nemirovski \cite{nesterov1989self,nesterov1994interior}, Alizadeh \cite{alizadeh1995interior,alizadeh1991combinatorial}, and others independently applied second-order interior-point methods to SDP. These methods generate a sequence of unconstrained subproblems during iterations and employ Newton's method for approximate solutions. Theoretically, second-order interior-point methods can solve any SDP problem to arbitrary precision within polynomial time. However, each iteration requires solving a linear system involving the KKT system, making computational and memory costs prohibitively high for large-scale SDP problems. Nevertheless, second-order interior-point methods remain practical for solving medium and small-scale SDP problems, as evidenced by software implementations like SeDuMi \cite{sturm1999using}, MOSEK\cite{aps2019mosek}, and SDPT3 \cite{toh1999sdpt3}.

In recent years, the demand for large-scale SDP applications has continued to grow \cite{majumdar2020recent}. First-order algorithms have become a major research direction in this field, primarily due to their advantages of low iteration complexity and low storage overhead. For example, the Alternating Direction Method of Multipliers (ADMM) has been introduced to solve large-scale SDP problems, leading to the development of SDPAD \cite{wen2010alternating}. The 2EBD \cite{monteiro2014first} addresses SDP problems with two-easy-block structures using first-order block decomposition techniques. Subsequently, the convergent three-block ADMM method based on symmetric Gauss-Seidel \cite{sun2015convergent} has been developed, along with algorithms for solving the homogeneous self-dual embedding problems of SDP within the ADMM framework \cite{o2016conic}. Another effective framework for solving SDP is the Augmented Lagrangian Method (ALM). SDPNAL \cite{zhao2010newton} employs the semi-smooth Newton method to construct a Newton-CG ALM, and its extended MATLAB package SDPNAL+ \cite{yang2015sdpnal+} can handle SDP problems with boundary constraints, demonstrating excellent numerical performance. Similarly, the SSNSDP algorithm \cite{li2018semismooth} is designed to provide a solution framework based on the equivalence between Douglas-Rachford Splitting (DRS) and ADMM. To address memory issues within the AL framework, optimal-storage algorithms such as SketchyCGAL \cite{yurtsever2021scalable} (combining sketching techniques \cite{tropp2017practical} and conditional gradient methods) and CSSDP \cite{ding2021optimal} (which solves the dual problem via first-order methods and recovers the primal solution) have been developed. Exploring the application of chordal sparsity \cite{tang2024exploring,zheng2020chordal} and Facial Reduction \cite{hu2023facial,permenter2018partial} techniques in SDP remains an important research direction.  

An alternative idea is to factorize the matrix variable \(X = UU^{T}\), where \(U \in \mathbb{R}^{n\times r}\) with \(r\ll n\), and reformulate the primal SDP (\ref{eq:Psdp}) as:
\begin{equation}\label{eq:lowrank}
\underset{U\in \mathbb{R}^{n \times r}}{\min}\quad\langle C,UU^{T} \rangle\quad \text{s.t.}\quad \mathcal{A}(UU^{T})=b.
\end{equation}
A wide variety of nonlinear programming methods can be applied to optimize the equation (\ref{eq:lowrank}) with respect to the variable \(U\). The theoretical studies of Barvinok \cite{barvinok1995problems} and Pataki \cite{burer2003nonlinear} indicate that the primal SDP (\ref{eq:Psdp}) always has a solution with rank \(r\), where \(\frac{r(r + 1)}{2} \leq m\). Inspired by the existence of low-rank solutions to SDP, Burer and Monteiro \cite{burer2003nonlinear} proposed to solve the optimization problem (\ref{eq:lowrank}) with a limited-memory BFGS algorithm. A further development in low-rank methods for SDP that utilizes the idea of matrix factorization is the Riemannian optimization method \cite{journee2010low,wen2013feasible,boumal2015riemannian,wang2023decomposition,tang2024feasible}. These methods are based on the assumption that the feasible set of the equation (\ref{eq:lowrank}) is a Riemannian manifold. 

A related yet distinct strategy adopting the low-rank concept approximates the positive semidefinite constraint \( X \in \mathbb{S}_n^+ \) by \( X = V S V^\top \), where \( V \in \mathbb{R}^{n \times l} \), \( S \in \mathbb{S}_+^l \), and \( r \leq l \ll n \). This is one of the core ideas in the design of spectral bundle methods \cite{helmberg2000spectral,ding2023revisiting}. Another class of methods similar to spectral bundle methods is the polyhedral bundle method \cite{krishnan2006unifying,Gu2017optimization}. The polyhedral bundle method is an iterative method that modifies the dual SDP (\ref{eq:Dsdp}): it transforms the semidefinite constraint into a polyhedral constraint composed of multiple linear constraints, and adds a quadratic term centered at a given reference point to the objective function. In this way, we convert an SDP problem into a quadratic programming problem, which enables efficient solution of the problem. Its main challenge lies in how to continuously revise the polyhedral constraint during iterations, with the goal of better approximating the semidefinite constraint while minimizing the number of bundles as much as possible. Regarding the specific form of the polyhedral bundle method, we will describe the above content using the penalty function method in subsequent sections.  
\subsection{Our contribution}
The core advantage of the polyhedral bundle method lies in its ability to convert SDP into a series of quadratic programming subproblems, and the computational cost of generating and solving these quadratic programming subproblems is far lower than that of the quadratic SDP subproblems relied on by the spectral bundle method. However, under the classical polyhedral bundle framework, if there are no effective constraints on the number of bundles, the lower approximation model may cause the number of bundles to increase continuously with iterations in its effort to enhance its own approximation capability during the iteration process, which in turn leads to a decline in the solution efficiency of the subproblems. Therefore, setting a reasonable upper bound \( l_{\max} \) for the number of bundles in the polyhedral bundle method and minimizing the number of constraints during the iteration process as much as possible are crucial for achieving efficient solution. Based on this basic idea, we first propose a modified quadratic programming subproblem, which not only avoids the singularity that may occur during the solution process but also reduces the number of constraints required in the solution process to a certain extent.

Regarding the setting of the upper bound \( l_{\max} \) for the number of bundles, we established a heuristic relationship between the upper bound of the number of bundles and the rank of the optimal solution \( r^{\star} \) for standard SDP through numerical experiments. Setting the upper bound of the number of bundles as \( l_{\max} = \frac{1}{2}r^{\star}(r^{\star}+1)+r^{\star} \) can effectively reduce the dimension of decision variables in the quadratic programming subproblems and improve the solution efficiency. If the rank of the optimal solution \( r^{\star} \) is unknown, the rank of the optimal solution \( r^{\star} \) can also be predicted in the early stage of iterations using the rank prediction algorithm we proposed.

Our algorithm performs well in both medium-and large-scale random sparse SDP problems and Max-Cut problems, especially for problems with low condition numbers. Our research shows that for problems with high condition numbers, insufficient number of bundles will make it difficult for the algorithm to converge. We believe the reason for this phenomenon is that a high condition number significantly increases the difficulty of the lower approximation model in approximating the original problem, thus, more bundles are needed to enhance the model's approximation capability. Therefore, by appropriately increasing the upper bound of the number of bundles or raising the rank used in the iteration process, the algorithm can solve the problem to the required precision within a given time or number of iterations.
\subsection{Notation}
We will use the Frobenius norm \(\|\cdot\|_F\), the \(\ell_2\) operator norm \(\|\cdot\|_{\text{op}}\), the norm \(\|\cdot\|_2\), and \(\|\cdot\|\) induced by the canonical inner product of the underlying real vector space in this paper. Let \( E \) be a non-empty set of points in \( \mathbb{R}^d \); then the convex hull of \( E \) (denoted as \( \text{conv}(E) \)) is defined as the smallest convex set in \( \mathbb{R}^d \) that contains \( E \). Regardless of whether \(a\) and \(b\) are vectors or matrices, assuming they have the same number of rows, we denote their horizontal concatenation as \([a, b]\); similarly, assuming they have the same number of columns, we denote their vertical concatenation as \([a; b]\). We consider that the proper closed convex function function \(F(y)\) is \(M\)-Lipschitz continuous, with \(M<\infty\):
\begin{equation*}
|F(x)-F(y)| \leq M\| x -y\| \text{, for all }x,y \in \mathbb{R}^{m}. 
\end{equation*}
For a symmetric matrix \(A \in \mathbb{S}^{n}\), we define:  
\begin{equation*}
\text{svec}(A) = \left[ a_{11}, \sqrt{2}a_{21}, \cdots, \sqrt{2}a_{n1}, a_{22}, \sqrt{2}a_{32}\cdots, a_{nn} \right]^{T} \in \mathbb{R}^{\frac{n(n+1)}{2}}.
\end{equation*}
That is, the lower triangular elements of \(A\) are stacked column-wise into a column vector, where the diagonal elements remain unchanged and the off-diagonal elements are multiplied by \(\sqrt{2}\). Under this definition, if \(A, B \in \mathbb{S}^{n}\), then the inner product \(\langle A, B \rangle = \text{svec}(A)^{T} \text{svec}(B)\).

Suppose matrices \(A\) and \(B\) are both of size \(m \times n\). The hadamard product \(A \odot B\) satisfies \((A \odot B)_{i,j} = a_{i,j}b_{i,j}\). If matrix \(C\) is of size \(m \times n\) and matrix \(D\) is of size \(mk \times n\), we define:  
\[
\text{broadcast}(C) \odot D = \left[ \underbrace{C;\cdots;C}_{k \text{ times}} \right] \odot D.
\]
\subsection{Organization}
The structure of the subsequent sections in this paper is arranged as follows. In Section \ref{sec:2}, we review the general bundle method, the spectral bundle method, the polyhedral bundle method, and the modified subproblem we proposed based on the polyhedral bundle method. In Section \ref{sec:3}, we mainly present the update process of the lower approximation model, the stopping criterion, the rank update strategy, and the complete iterative algorithm for the polyhedral bundle method. In Section \ref{sec:4}, we prove the convergence of the polyhedral bundle method for solving SDP. In Section \ref{sec:5}, we present the performance of the proposed method on random sparse SDP and the Max-Cut problem. In Section \ref{sec:6}, we summarize the entire paper and outline future research directions.
\section{Bundle method}\label{sec:2}
In this section, we first introduce the general bundle method and the exact penalty function corresponding to dual SDP (\ref{eq:Dsdp}). Then, we introduce two methods, the spectral bundle method and the polyhedral bundle method, which are derived by applying the bundle method to SDP. Finally, we present the modified subproblem proposed based on the polyhedral bundle method.
\subsection{Proximal bundle method}
In this subsection, we overview proximal bundle methods \cite{feltenmark2000dual} for solving unconstrained non-smooth convex minimization problems of the form
\begin{equation}\label{eq:F}
\underset{y\in \mathbb{R}^{m}}{\min}F(y),
\end{equation}
where \(F : \mathbb{R}^{m} \rightarrow (-\infty, \infty]\) is a proper closed convex function that attains its minimum value, \(\inf F \), on some nonempty set. The proximal bundle method solves the following model at each iteration:
\begin{equation}\label{eq:bundle}
z_{k+1}\in \underset{z\in \mathbb{R}^{m}}{\arg\min }\text{ }F_{k}(z)+\frac{1}{2t}\|z-y_{k}\|^{2},
\end{equation}
where $t>0$. The proximal bundle method updates the iterate $y_k$ only when the objective value decrease $F(\cdot)$ meets a fraction of the decrease predicted by the approximate model $F_k(\cdot)$. Given $0 < \beta < 1$, if
\begin{equation}\label{eq:decrease}
\beta(F(y_{k})-F_{k}(z_{k+1})) \leq F(y_{k})-F(z_{k+1}),
\end{equation}
one perform a descent step by setting $y_{k+1} = z_{k+1}$; otherwise, a null step keeps $y_{k+1} = y_{k}$. In either case, the subgradient $g_{k+1} \in \partial F(z_{k+1})$ at the new point $z_{k+1}$ is used to update the lower approximation $F_{k+1}(y)$.

If the new lower approximation model $F_{k+1}(y)$ satisfies the following three properties\cite{liao2023overview}, then the proximal bundle method (\ref{eq:bundle}) converges:
\begin{enumerate}
    \item Minorant:\text{ }the function $F_{k+1}(y)$ is a lower bound on $F(y)$
        \begin{equation}\label{eq:cond1}
        F_{k+1}(y)\leq F(y),\forall y\in \mathbb{R}^{m}.
        \end{equation}
    \item Subgradient lower bound:\text{ }$F_{k+1}(y)$ is lower bounded by the linearization given by some subgradient $g_{k+1} \in \partial F(z_{k+1})$ computed after (\ref{eq:bundle})
        \begin{equation}\label{eq:cond2}
        F_{k+1}(y)\geq F(z_{k+1})+\langle g_{k+1},y-z_{k+1}\rangle,\forall y\in \mathbb{R}^{m}.
        \end{equation}
    \item Model subgradient lower bound:\text{ }$F_{k+1}(y)$ is lower bounded by the linearization of the model  $F_{k}(y)$ given by the subgradient $s_{k+1} := \frac{1}{t}(y_{k}-z_{k+1}) \in \partial F_{k}(z_{k+1})$, i.e.
        \begin{equation}\label{eq:cond3}
        F_{k+1}(y)\geq F_{k}(z_{k+1})+\langle s_{k+1},y-z_{k+1}\rangle,\forall y\in \mathbb{R}^{m}.
        \end{equation}
\end{enumerate}
Algorithm \ref{alg:1} presents the iterative procedure of the General Bundle Method.
\begin{algorithm}[H]
\caption{General Bundle Method}\label{alg:1}
\begin{algorithmic}[1]
    \Require $ F(y);\text{Parameters } t>0,\beta \in (0,1);\text{maxiter}$
    \For{\(k=1:\text{maxiter}\)}
        \State \text{Solve (\ref{eq:bundle}) to obtain a candidate iterate} \(z_{k+1}\).
        \If{\ref{eq:decrease} is satisfied}
            \State Set \(y_{k+1}=z_{k+1}\)
        \Else
            \State Set \(y_{k+1}=y_{k}\)
        \EndIf
        \State Update approximation model \(F_{k+1}(y)\) without violating (\ref{eq:cond1}) to (\ref{eq:cond3}).
        \If{stopping criterion}
            \State Quit.
        \EndIf
    \EndFor
\end{algorithmic}
\end{algorithm}
\subsection{Exact penalization of dual SDP}
To apply the proximal bundle method to SDP, we consider the following non-smooth convex penalty function form for dual SDP (\ref{eq:Dsdp}),
\begin{equation}\label{eq:penalty}
\underset{y\in \mathbb{R}^{m}}{\min}F(y)=-b^{T}y+\rho\max\{\lambda_{\max}(\mathcal{A}^{T}y-C),0\}.
\end{equation}

\begin{theorem}\label{theorem:1}
The penalized non-smooth formulation (\ref{eq:penalty}) is an exact penalization for dual SDP (\ref{eq:Dsdp}). Let \(P^{\star}\) denote the set of primal optimal solutions, and let \(\rho > \mathcal{D}_{X^{\star}} = \max\limits_{X^{\star} \in P^{\star}} \text{Tr}(X^{\star})\). A point \((y^{\star}, S^{\star})\) is an optimal solution of the dual SDP (\ref{eq:Dsdp}) if and only if it is an optimal solution of (\ref{eq:penalty}).
\end{theorem}
Theorem \ref{theorem:1} shows that the penalty function (\ref{eq:penalty}) is an exact penalty function for the dual problem (\ref{eq:Dsdp}). Under this theorem, the selection of the penalty parameter \(\rho\) must be based on \(\mathcal{D}_{X^{\star}}\), and this selection depends on \(\mathcal{D}_{X^{\star}}\) as follows. If \(\mathcal{D}_{X^{\star}}\) is known, the value of \(\rho\) that satisfies \(\rho > \mathcal{D}_{X^{\star}}\) can be directly selected. If \(\mathcal{D}_{X^{\star}}\) is unknown, judgment must be made in combination with the structure of the problem's solution \(X^{\star}\), and this judgment needs to consider the structure of \(X^{\star}\) specifically. If \(X^{\star}\) has no special structure, a larger \(\rho\) can be selected to ensure \(\rho > \mathcal{D}_{X^{\star}}\). If \(X^{\star}\) has a special structure, \(\mathcal{D}_{X^{\star}}\) can be precomputed first, and then an appropriate \(\rho\) can be selected based on this. To deeply study Theorem \ref{theorem:1} and its proof, one can refer to the literature \cite{ding2023revisiting}.

Constructing the lower approximation model \( F_{k}(y) \) requires calculating the subgradient of \( F(y) \), let \(S = C-\mathcal{A}^{T}y\), The subdifferential of the penalty function \(F(y)\)\cite{ding2021optimal} is given by
\begin{equation*}
\partial F(y)=
\begin{cases} 
-b+\textbf{conv}\{\rho\mathcal{A}(vv^{T})|Sv=\lambda_{\min}(S)v\}, &\lambda_{\min}(S)< 0, \\
-b,       & \lambda_{\min}(S)> 0, \\
-b+\beta\textbf{conv}\{\rho\mathcal{A}(vv^{T})|Sv=\lambda_{\min}(S)v,\beta\in[0,1]\}, & \lambda_{\min}(S)= 0.
\end{cases}
\end{equation*}
Here we allow the smallest eigenvalue \(\lambda_{\min}(S)\) of \(S\) to have multiplicity \( r\), where \(v\) is the eigenvector corresponding to \(\lambda_{\min}(S)\). 
\subsection{Spectral bundle method}
To obtain a lower approximation model for \(F(y)\) and derive the corresponding spectral bundle method, one can rewrite penalty function (\ref{eq:penalty}) based on the following facts:
\begin{equation*}
\max\{\lambda_{\max}(\mathcal{A}^{T}y-C),0\}=\underset{X\in \mathbb{S}^{n}_{+},\text{Tr}(X)\leq 1}{\max}\langle X,\mathcal{A}^{T}y-C\rangle.
\end{equation*}
Thus, the penalty problem (\ref{eq:penalty}) can be rewritten as:
\begin{equation}\label{eq:sub1}
\begin{aligned}
\underset{y\in \mathbb{R}^{m}}{\min}F(y)&=-b^{T}y+\rho\max\{\lambda_{\max}(\mathcal{A}^{T}y-C),0\}\\
&=-b^{T}y+\rho\underset{X\in \mathbb{S}^{n}_{+},\text{Tr}(X)\leq 1}{\max}\langle X,\mathcal{A}^{T}y-C\rangle\\ 
&=-b^{T}y+\underset{X\in \mathbb{S}^{n}_{+},\text{Tr}(X)\leq \rho}{\max}\langle X,\mathcal{A}^{T}y-C\rangle\\ 
&=-b^{T}y+\underset{X\in \mathcal{W}}{\max}\langle X,\mathcal{A}^{T}y-C\rangle,
\end{aligned}
\end{equation}
where \(\mathcal{W}=\{X\in \mathbb{S}^{n}_{+},\text{Tr}(X)\leq \rho\}\), However, problem (\ref{eq:sub1}) remains challenging to solve directly. Nevertheless, one can derive a lower approximation model for \( F(y) \) by selecting an appropriate subset of \( \mathcal{W} \). Specifically, if \(\mathcal{W}_{1} = \{ PSP^T \mid P \in \mathbb{R}^{n \times r},\, P^T P = I,\, S \in \mathbb{S}^r_+,\, \text{Tr}(S) \leq \rho \} \subset \mathcal{W}\) is chosen, then the lower approximation model of the \(k\)-th iteration can be written as:
\begin{equation*}
F_{\mathcal{W}_{1,k}}(y)=-b^{T}y+\underset{X\in \mathcal{W}_{1,k}}{\max}\langle X,\mathcal{A}^{T}y-C\rangle\leq-b^{T}y+\underset{X\in \mathcal{W}}{\max}\langle X,\mathcal{A}^{T}y-C\rangle=F(y),\forall y\in \mathbb{R}^{m}.
\end{equation*}
To enhance the approximation capability of the lower approximation model, an additional term can be added on the basis of \(\mathcal{W}_1\) to aggregate historical gradient information and avoid information loss, and its impact on computational efficiency is negligible. Consequently, \(\mathcal{W}_2\) is defined as \(\{ \eta W + PSP^T \mid W \in \mathbb{S}^n_+, \, \text{Tr}(W) = 1, \, P \in \mathbb{R}^{n \times r}, \, P^T P = I, \, \eta \geq 0, \, S \in \mathbb{S}^r_+, \, \eta + \text{Tr}(S) \leq \rho \}\). When \(\eta = 0\), \(\mathcal{W}_1 = \mathcal{W}_2\), thus \(\mathcal{W}_1\) is a special case of \(\mathcal{W}_2\). \(\mathcal{W}_2\) has stronger approximation capability and is a fundamental element in the derivation of the spectral bundle method, which satisfies \(\mathcal{W}_1 \subset \mathcal{W}_2 \subset \mathcal{W}\). Moreover, at the \(k\)-th iteration, it satisfies:
\begin{equation*}
F_{\mathcal{W}_{1,k}}(y)\leq F_{\mathcal{W}_{2,k}}(y)\leq F(y),\forall y\in \mathbb{R}^{m}.
\end{equation*}
Therefore, at the k-th iteration, the spectral bundle method solves the following problem:
\begin{equation}\label{eq:sub2}
\begin{aligned}
\underset{z\in\mathbb{R}^{m}}{\min}\hat{F}_{\mathcal{W}_{2,k}}(z)&=\underset{z\in\mathbb{R}^{m}}{\min}F_{\mathcal{W}_{2,k}}(z)+\frac{1}{2t_{k}}\|z-y_{k}\|^{2}\\
&=\underset{z\in\mathbb{R}^{m}}{\min}-b^{T}z+\underset{\eta\geq0,S\in \mathbb{S}^r_+,\eta+\text{Tr}(S) \leq \rho}{\max}\langle\eta W_{k} + P_{k}SP_{k}^T, \mathcal{A}^{T}z-C\rangle+\frac{1}{2t_{k}}\|z-y_{k}\|^{2}\\
&=\underset{z\in\mathbb{R}^{m}}{\min}\underset{\eta\geq0,S\in \mathbb{S}^r_+,\eta+\text{Tr}(S) \leq \rho}{\max}-b^{T}z+\langle\eta W_{k} + P_{k}SP_{k}^T, \mathcal{A}^{T}z-C\rangle+\frac{1}{2t_{k}}\|z-y_{k}\|^{2}\\ 
&=\underset{\eta\geq0,S\in \mathbb{S}^r_+,\eta+\text{Tr}(S) \leq \rho}{\max}\underset{z\in\mathbb{R}^{m}}{\min}-b^{T}z+\langle\eta W_{k} + P_{k}SP_{k}^T, \mathcal{A}^{T}z-C\rangle+\frac{1}{2t_{k}}\|z-y_{k}\|^{2}\\ 
&=\underset{\tilde{W}\in \mathcal{W}_{2,k}}{\min}\frac{t_{k}}{2}\| \mathcal{A}(\tilde{W})-b\|^{2}+\langle \tilde{W}, C-\mathcal{A}^{T}y_{k}\rangle+b^{T}y_{k}.
\end{aligned}
\end{equation} 
The above problem is a quadratic SDP problem, where the derivation of the last equality utilizes the optimality condition for the inner minimization with respect to the variable \(z\): \(z - y_k = t_k(b - \mathcal{A}(\eta W_k + P_k S P_k^T))\). The actual variables of Problem (\ref{eq:sub2}) are \((\eta, S)\). By solving Problem (\ref{eq:sub2}), the solution \((\eta_k, S_k)\) is obtained, and thus \(\tilde{W}_k = \eta_k W_k + P_k S_k P_k^T\), while \(z_{k+1} = y_k + t_k(b - \mathcal{A}(\eta W_k + P_k S P_k^T)) = y_k + t_k(b - \mathcal{A}(\tilde{W}_k))\). Although Problem (\ref{eq:sub2}) is a quadratic SDP, since its positive semidefinite variable \(S \in \mathbb{S}_+^r\) and \(r\) is small, it can be efficiently solved using interior point methods, such as Mosek. The update of \(F_{\mathcal{W}_{2,k+1}}(y)\) can be referred to \cite{helmberg2000spectral,ding2023revisiting,liao2023overview}. 
\subsection{Polyhedral Bundle method}
Although the dimension of the positive semidefinite variable is usually small, the polyhedral bundle method may yield better results for some special problems. By analogy with \(\mathcal{W}_1\) and \(\mathcal{W}_2\), \(\mathcal{W}_3 = \{ P\text{diag}(x)P^T \mid P \in \mathbb{R}^{n \times l},\, x \geq 0,\, \mathbf{1}^T x \leq \rho \}\) and \(\mathcal{W}_4 = \{ \eta W + P\text{diag}(x)P^T \mid P \in \mathbb{R}^{n \times l},\, \eta \geq 0,\, x \geq 0,\, \eta + \mathbf{1}^T x \leq \rho \}\) **are defined**, which satisfy \(\mathcal{W}_3 \subset \mathcal{W}_4 \subset \mathcal{W}\). Here, the number of bundles \(l\) is crucial to the problem, and its impact on the problem is reflected in two aspects. On the one hand, if the number of bundles is insufficient, the approximation capability of the model will be poor. On the other hand, if the number of bundles is excessive, the computational complexity of solving the subproblem will increase. Thus, reasonably controlling the upper bound of the number of bundles for different problems is an issue worth exploring. By adopting \(\mathcal{W}_4\), the following lower approximation model at the \(k\)-th iteration can be derived:
\begin{equation*}
F_{\mathcal{W}_{4,k}}(y)=-b^{T}y+\underset{\eta \geq 0, x \geq 0, \eta + \mathbf{1}^{T}x \leq \rho}{\max}\langle \eta W_{k} + P_{k}\text{diag}(x)P_{k}^{T}, \mathcal{A}^{T}y-C\rangle\leq F(y),\forall y\in \mathbb{R}.
\end{equation*}
Therefore, at the k-th iteration, the Polyhedral bundle method solves the following problem:
\begin{equation}\label{eq:sub3}
\begin{aligned}
\underset{z\in\mathbb{R}^{m}}{\min}\hat{F}_{\mathcal{W}_{4,k}}(z)&=
\underset{z\in\mathbb{R}^{m}}{\min}F_{\mathcal{W}_{4,k}}(z)+\frac{1}{2t_{k}}\|z-y_{k}\|^{2}\\
&=\underset{z\in\mathbb{R}^{m}}{\min}-b^{T}z+\underset{\eta \geq 0, x \geq 0, \eta + \mathbf{1}^{T}x \leq \rho}{\max}\langle\eta W_{k} + P_{k}\text{diag}(x)P_{k}^{T}, \mathcal{A}^{T}z-C\rangle+\frac{1}{2t_{k}}\|z-y_{k}\|^{2}\\
&=\underset{z\in\mathbb{R}^{m}}{\min}\underset{\eta \geq 0, x \geq 0, \eta + \mathbf{1}^{T}x \leq \rho}{\max}-b^{T}z+\langle\eta W_{k} + P_{k}\text{diag}(x)P_{k}^{T}, \mathcal{A}^{T}z-C\rangle+\frac{1}{2t_{k}}\|z-y_{k}\|^{2}\\ 
&=\underset{\eta \geq 0, x \geq 0, \eta + \mathbf{1}^{T}x \leq \rho}{\max}\underset{z\in\mathbb{R}^{m}}{\min}-b^{T}z+\langle\eta W_{k} + P_{k}\text{diag}(x)P_{k}^{T}, \mathcal{A}^{T}z-C\rangle+\frac{1}{2t_{k}}\|z-y_{k}\|^{2}\\ 
&=\underset{\tilde{W}\in \mathcal{W}_{4,k}}{\min}\frac{t_{k}}{2}\| \mathcal{A}(\tilde{W})-b\|^{2}+\langle \tilde{W}, C-\mathcal{A}^{T}y_{k}\rangle+b^{T}y_{k}.
\end{aligned}
\end{equation}
The above problem is a quadratic programming problem, where the derivation of the last equality utilizes the optimality condition for the inner minimization with respect to the variable \(z\): \(z - y_k = t_k\left(b - \mathcal{A}\left(\eta W_k + P_k \text{diag}(x) P_k^T\right)\right)\). The actual variables of Problem (\ref{eq:sub3}) are \(\eta\) and \(x\). By solving Problem (\ref{eq:sub3}), the solution \((\eta_k, x_k)\) is obtained, and thus \(\tilde{W}_k = \eta_k W_k + P_k \text{diag}(x_k) P_k^T\), while \(z_{k+1} = y_k + t_k\left(b - \mathcal{A}\left(\eta W_k + P_k \text{diag}(x_k) P_k^T\right)\right) = y_k + t_k\left(b - \mathcal{A}(\tilde{W}_k)\right)\). \(u = [\eta; x]\) is defined, where \(\eta \in \mathbb{R}_{+}\), \(x \in \mathbb{R}^{l}_{+}\), \(u \in \mathbb{R}^{l+1}_{+}\) and \(\mathbf{1}^T u \leq \rho\). Suppose \(v_{k,1}, \ldots, v_{k,l}\) are the columns of \(P_k\). By simplifying the above equation, one can obtain
\begin{equation*}
\bar{a}_k = \langle C, W_k \rangle \in \mathbb{R}\text{  and  }  \bar{B}_k = \mathcal{A}(W_k) \in \mathbb{R}^{m}. 
\end{equation*}
For \( i = 1, \ldots, l \), we have
\begin{equation}\label{eq:cons1}
\hat{a}_{k}=[\hat{a}_{k,1};\cdots;\hat{a}_{k,l}]=[\langle C,v_{k,1}(v_{k,1})^{T}\rangle;\cdots;
\langle C,v_{k,l}(v_{k,l})^{T}\rangle],
\end{equation}
and
\begin{equation}\label{eq:cons2}
\hat{B}_{k}=[B_{k,1},\cdots,B_{k,l}]=[\mathcal{A}(v_{k,1}(v_{k,1})^{T}),\cdots,\mathcal{A}(v_{k,l}(v_{k,l})^{T})].
\end{equation}
Therefore, we have \( a_{k} = [\bar{a}_k; \hat{a}_{k}] \) and \( B_{k} = [\bar{B}_k, \hat{B}_{k}] \). Using the above definitions, we can transform Problem (\ref{eq:sub3}) into the following quadratic programming problem:
\begin{equation*}
\underset{z\in\mathbb{R}^{m}}{\min}\hat{F}_{\mathcal{W}_{4,k}}(y)=\underset{u\geq0,\mathbf{1}^T u \leq \rho}{\min}\frac{1}{2}u^{T}(t_{k}B_{k}^{T}B_{k})u+(a_{k}-B_{k}^{T}(y_{k}+t_{k}b))^{T}u.
\end{equation*}
Then \(z_{k+1} = y_{k} + t_{k}(b - B_{k}u_{k})\), so the above iteration can be transformed into:
\begin{equation}\label{eq:sub4}
\begin{aligned}
u_{k}&=\underset{u\geq0,\mathbf{1}^T u \leq \rho}{\arg\min}\frac{1}{2}u^{T}(t_{k}B_{k}^{T}B_{k})u+(a_{k}-B_{k}^{T}(y_{k}+t_{k}b))^{T}u\\ 
z_{k+1}&=y_{k}+t_{k}(b-B_{k}u_{k}).
\end{aligned}
\end{equation}
The advantage of Problem (\ref{eq:sub4}) lies in that it is a quadratic programming problem with low solution difficulty. However, its core difficulty lies in controlling the number of bundles \( l \). If \( l \) is too small, the approximation capability of the lower approximation model will be insufficient, and it may fail to reach the preset convergence accuracy within the given time or number of iterations. If \( l \) is too large, it will reduce the computational efficiency of the subproblem, thereby prolonging the overall solution time. In addition, if no restriction is imposed on the number of bundles \( l \), \( l \) will continue to increase with the iteration process, further continuously reducing the solution efficiency of the subproblem. Therefore, the number of bundles \(l\) should be balanced to ensure both convergence accuracy and solution efficiency. Setting an appropriate upper bound \(l_{\max}\) helps enhance the overall computational efficiency without compromising accuracy.
\subsection{The modified polyhedral bundle method}\label{subsec2.5}
In practical experiments, we need to make the following modifications to problem (\ref{eq:sub4}):
\begin{equation}\label{eq:sub5}
\begin{aligned}
u_{k}&=\underset{u\geq0,\mathbf{1}^T u \leq \rho}{\arg\min}\frac{1}{2}u^{T}(t_{k}B_{k}^{T}B_{k}+\xi I)u+(a_{k}-B_{k}^{T}(y_{k}+t_{k}b))^{T}u\\ 
z_{k+1}&=y_{k}+t_{k}(b-B_{k}u_{k}),
\end{aligned}
\end{equation}
where \(\xi \geq 0\). The advantage of subproblem (\ref{eq:sub5}) over subproblem (\ref{eq:sub4}) lies in that the quadratic term matrix \(t_{k}B_{k}^{T}B_{k}\) is a positive semi-definite matrix, which may thus be singular, leading to numerical instability. Therefore, adding a small regularization term \(\xi I\) can ensure the stability of the solution to a certain extent, and then methods such as the interior point method\cite{aps2019mosek}, active set method\cite{gu2024random}, or a partial first-order affine scaling method\cite{gu2019partial} can be used for solving. The parameter \(\xi\) should not be too large, otherwise the solution will deviate too much from that of the original problem. If it is too small, it may cause problems in the adjustment of the step size parameter \(t_{k}\). We present the dual problem of Problem (\ref{eq:sub5}):
\begin{equation}\label{eq:sub6}
\begin{aligned}
&\underset{y\in R^{m}}{\min}\frac{1}{2t_{k}}\|y-y^{k}-t_{k}b\|_{2}^{2}+\frac{\xi}{2t_{k}^{2}}(y-y^{k}-t_{k}b)^{T}B(B^{T}B)^{-2}B^{T}(y-y^{k}-t_{k}b)\\ 
&\text{  s.t. }B^{T}y+t_{k}^{-1}\xi(B^{T}B)^{-1}B^{T}(y-y^{k}-t_{k}b)-a\leq0
\end{aligned}
\end{equation}
\begin{figure}
 \centering
 \subfigure[\(\xi=0\)]{\label{215.1} 
 \includegraphics[width=2.7in]{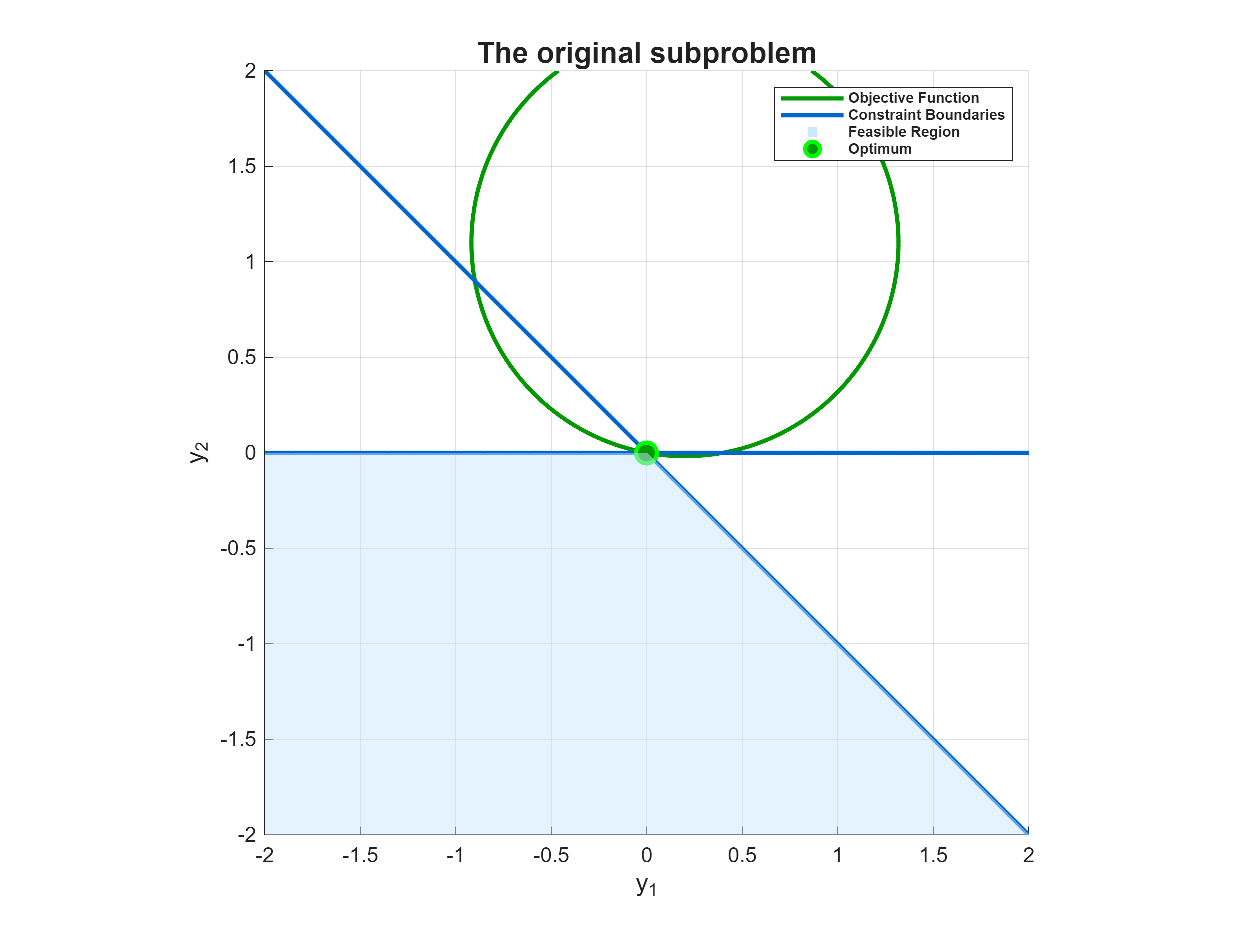}}
 \hspace{0.01in}
 \subfigure[\(\xi=0.1\)]{\label{215.2} 
 \includegraphics[width=2.7in]{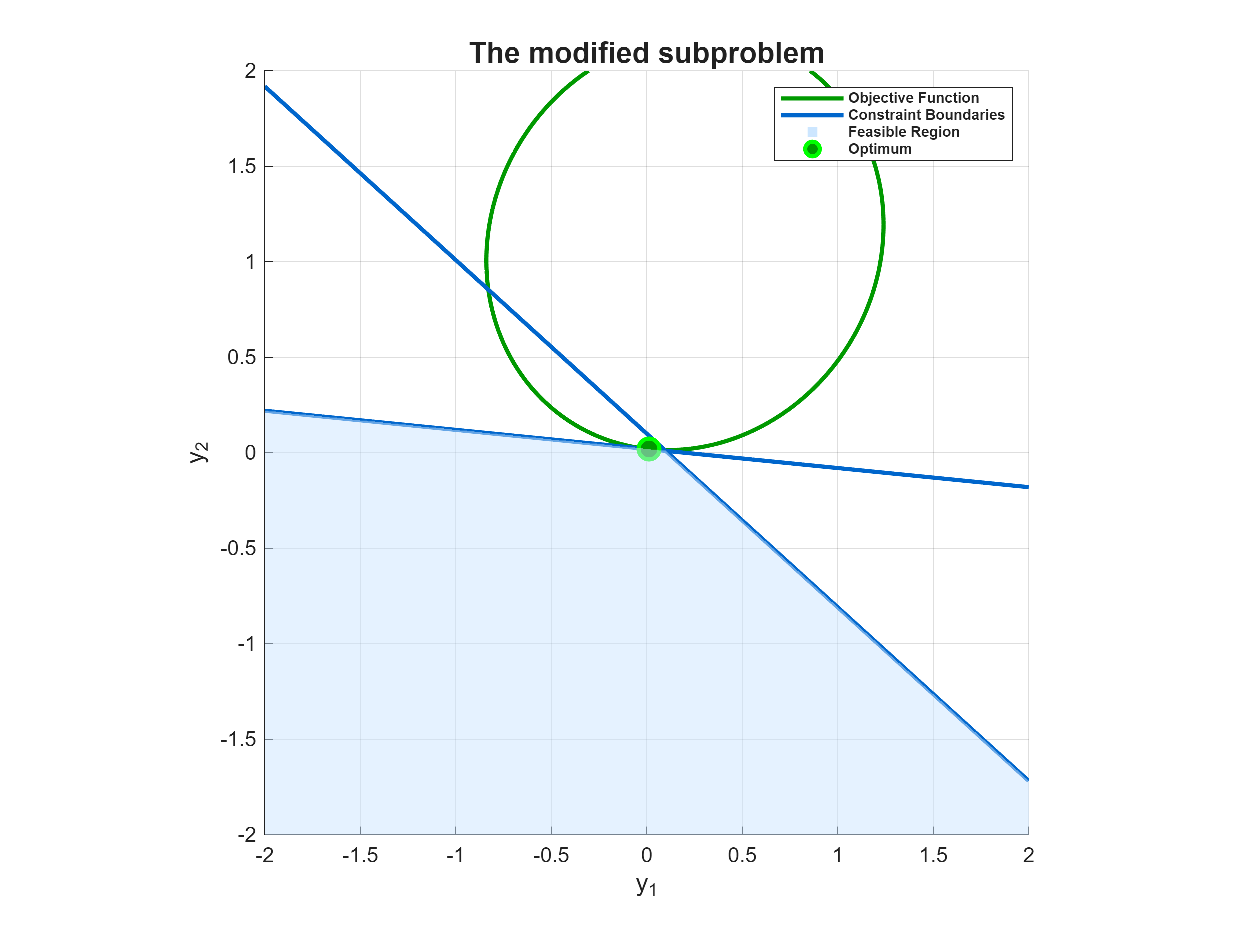}}
 \caption{Relationship between the properties of the solution to Subproblem (\ref{eq:sub6}) and Parameter \(\xi\), where the problem parameters are set as follows: \(t_{k} = 1,y_{k} = [0.1; 1],b = [0.1;0.1],B = [0,1;1,1],a = [0;0].\) The blue lines represent the constraints, the light blue area denotes the feasible region, and the green curve is the contour curve corresponding to the objective function and the optimal solution (marked by the green circle).}\label{fig:4} 
\end{figure}

From Figure \ref{fig:4}, it can be seen that when the parameter \(\xi = 0\), both constraints at the optimal solution are satisfied but not strictly satisfied. Therefore, both of these constraints may be retained. On the other hand, when the parameter \(\xi = 0.1\), one constraint at its optimal solution is strictly satisfied. When a certain threshold is set, this constraint can be removed. Thus, for Subproblem (\ref{eq:sub5}), the reasonable setting of parameter \(\xi\) can reduce the number of constraints in the iteration process to a certain extent, thereby improving the solution efficiency of the subproblem.  

Here, we construct a simple SDP problem to demonstrate the necessity of setting parameter \(\xi\). The parameters are set as follows: \(n = 3\), \(m = 3\), and \(r^{\star} = 2\) (where the rank of the optimal solution for the primal variable is 2). Since the constructed problem satisfies the strict complementary slackness condition, the rank of the optimal solution for the dual variable is 1. Experiments are conducted with parameters \(\xi = 0\), \(10^{-12}\), \(10^{-10}\), \(10^{-8}\), \(10^{-6}\), and \(10^{-4}\) respectively. The maximum number of iterations is set to 500. If numerical instability occurs, it is recorded as \(\textbf{Fail}\).  
\begin{table}
    \centering    
    \caption{The effect of parameter \(\xi\) on the number of iterations}\label{t1}
    \setlength\tabcolsep{0.9pt}
    \begin{tabular*}{\textwidth}{@{\extracolsep\fill}ccccccccc}
\toprule 
~&~&~& \multicolumn{6}{c}{number of iterations}  \\ 
\cmidrule{4-9}
n&m&\(r^{\star}\)&\(\xi=0\)&\(\xi=10^{-12}\)&\(\xi=10^{-10}\)&\(\xi=10^{-8}\)&\(\xi=10^{-6}\)&\(\xi =10^{-4}\)\\
\midrule    
3&3&2&\textbf{Fail} &10 &10 &12 &\textbf{Fail} &500 \\
\hline
\end{tabular*}
\end{table}

As shown in Table \ref{t1}, for this problem, numerical instability occurs when the term \(\xi I\) is not added. On the other hand, when \(\xi\) is too large, it will cause a significant deviation between the solution of the modified subproblem and the atomic problem, leading to a substantial increase in the number of iterations and failure to converge. However, a larger \(\xi\) can minimize the singularity of the quadratic term as much as possible. Without significantly affecting the number of iterations, in subsequent experiments, we adopt \(\xi=10^{-8}\) as the correction parameter, which not only ensures numerical stability but also guarantees that the deviation of the modified subproblem from the atomic problem can be neglected when the solution accuracy is relatively low.
\section{Algorithms}\label{sec:3}
In this section, we shall expound in detail on the following four aspects. First, it refers to the update of the lower approximation model corresponding to the modified subproblem (\ref{eq:sub5}). Second, it denotes the convergence criterion of the algorithm \ref{alg:3}. Third, it stands for the prediction algorithm \ref{alg:4} for the rank of the optimal solution. Fourth, it represents the complete iterative algorithm \ref{alg:3} for the modified subproblem (\ref{eq:sub5}). Among the aforementioned aspects, the update process of the lower approximation model corresponding to the modified subproblem (\ref{eq:sub5}) bears resemblance to that of the spectral bundle method.
\subsection{Lower approximation model update for modified polyhedral bundle method}
In the polyhedral bundle method, the core of updating the lower approximation model lies in the iterative update of variables \(W_{k}\) and \(P_{k}\). The variable \(W_{k}\) is used to aggregate historical gradient information, while \(P_{k}\) represents the currently active gradient information. Within the framework of the polyhedral bundle method, After solving the lower approximation model composed of \(W_{k}\) and \(P_{k}\), a new candidate point \(z_{k+1}\) can be obtained. Based on this new candidate point, the latest gradient information is incorporated into \(P_{k}\), while gradient information with relatively lower activity is aggregated into \(W_{k}\). Subsequently, new variables \(W_{k+1}\), \(P_{k+1}\), and the lower approximation model \(F_{\mathcal{W}_{4,k+1}}(y)\) are constructed. In the subsequent discussions, unless otherwise specified, we define \(F_{k+1}(y) =F_{\mathcal{W}_{4,k+1}}(y)\). Given that the rank \(r^{\star}\) of the optimal solution and the upper bound \(l_{\max}\) of the number of bundles are known, the specific update process of the polyhedral bundle method is as follows, with detailed explanations provided below.
    
Given \(W_k, P_k\), solving the lower approximation model (\ref{eq:sub5}) yields \(u_k = (\eta_k; x_k)\) and the candidate point \(z_{k+1}\), where \(\eta_k\) denotes the weight of historical gradient information, and each element \((x_k)_i\) in \(x_k\) represents the weight of the corresponding bundle \((P_k)_i\). The smaller its value, the smaller the contribution of the corresponding bundle to the model, and such a bundle should be aggregated into \(W_k\) to obtain \(W_{k+1}\). Meanwhile, the polyhedral bundle method computes the eigenvectors corresponding to the smallest \(r^{\star}\) eigenvalues of \(S_{k+1}=C-\mathcal{A}^T z_{k+1}\), and adds them to \(P_k\) to form \(P_{k+1}\). In this process, it is necessary to ensure that the number of bundles after update, \(l_{k+1}\), does not exceed \(l_{\max}\).

Since the number of gradient information added in each iteration is \(r^{\star}\), two cases need to be considered. In the first case, if \(l_{k}\leq l_{\max}-r^{\star}\), then \(l_{k + 1}\leq l_{\max}\) can be satisfied without aggregating gradients. In the second case, if \(l_{k}>l_{\max}-r^{\star}\), at least \(l_{k}-(l_{\max}-r^{\star})\) bundles with the smallest weights need to be aggregated to ensure \(l_{k+1}\leq l_{\max}\).

In the first case, we define the set of positions of the bundles to be aggregated as \(\{i|(x_{k})_{i}\leq\gamma_{1},i\in\{1,\cdots,l_{k}\}\}\), where \(\gamma_{1}\geq0\) is a relatively small parameter. In the second case, we define the set of positions of the bundles to be aggregated as the union of \(\{i|(x_{k})_{i}\leq\gamma_{2},i\in\{1,\cdots,l_{k}\}\}\) and the set composed of the positions corresponding to \(l_{k}-(l_{\max}-r^{\star})\) bundles with the smallest weights, where \(\gamma_{2}\geq0\) is a relatively small parameter, we set \(\gamma_1 = 10^{-6}\) and \(\gamma_{2}= 10^{-7}\).

Therefore, we divide all the bundles into two types: the set of positions of the bundles to be aggregated \(\bar{p}_{k}\) and the set of positions of the retained bundles \(\hat{p}_{k}\). These sets respectively correspond to the matrices \(\bar{P}_{k} = P_{k}(:, \bar{p}_{k})\) and \(\hat{P}_{k} = P_{k}(:, \hat{p}_{k})\). The matrix \(P_k(:, \bar{p}_k)\) is formed by the columns of \(P_k\) corresponding to the positions in \(\bar{p}_k\), and \(P_k(:, \hat{p}_k)\) is defined in the same way. The vectors \(\bar{x}_{k} = x_{k}(\bar{p}_{k})\) and \(\hat{x}_{k} = x_{k}(\hat{p}_{k})\) are composed of the elements in \(x_k\) at the positions corresponding to \(\bar{p}_k\) and \(\hat{p}_k\), respectively. By denoting the size of the set \(\bar{p}_{k}\) as \(d_{k}\), we can obtain \(W_{k+1}\) through the following formula:

\begin{equation}\label{eq:update1}
W_{k+1}=\frac{\eta_{k}W_{k}+\bar{P}_{k}\text{diag}(\bar{x}_{k})\bar{P}_{k}^{T}}{\eta_{k}+\textbf{1}^{T}\bar{x}_{k}}.
\end{equation}
Simultaneously, update \(\bar{a}_{k+1}\) and \(\bar{B}_{k+1}\):
\begin{equation}\label{eq:update2}
\bar{a}_{k+1}=\frac{\eta_{k}\bar{a}_{k}+a_{k}(\bar{p}_{k})^{T}\bar{x}_{k}}{\eta_{k}+\textbf{1}^{T}\bar{x}_{k}}
\text{ and }
\bar{B}_{k+1}=\frac{\eta_{k}\bar{B}_{k}+B_{k}(:,\bar{p}_{k})\bar{x}_{k}}{\eta_{k}+\textbf{1}^{T}\bar{x}_{k}}.
\end{equation}

If \(\eta_{k}\) and \(\bar{x}_{k}\) are both zero, then \(W_{k+1}=0\), \(\bar{a}_{k+1}=0\), \(\bar{B}_{k+1}=0\). Meanwhile, denote the number of bundles at this time as \(l_{k+\frac{1}{2}} = l_{k} - d_{k}\).

Calculate the smallest \(r^{\star}\) eigenvalues and their corresponding eigenvectors \(V_{k+1}\) of \(S_{k+1}=C-\mathcal{A}^{T}z_{k+1}\), and then construct the bundles required for the \((k+1)\)-th iteration:
\begin{equation}\label{eq:update3}
P_{k+1}=[\hat{P}_{k},V_{k+1}]. 
\end{equation}
Update the number of bundles \(l_{k+1} = l_{k+\frac{1}{2}} + r^{\star}\), compute the data \(a_{k+1}^\star, B_{k+1}^\star\) corresponding to \(V_{k+1}\) via Formula (\ref{eq:cons1}) and (\ref{eq:cons2}), and then update \(\hat{a}_{k+1}, a_{k+1}\):
\begin{equation}\label{eq:update4}
\hat{a}_{k+1} = [a_{k}(\hat{p}_{k});a_{k+1}^{\star}], a_{k+1} = [\bar{a}_{k+1}; \hat{a}_{k+1}],
\end{equation}
and \(\hat{B}_{k+1}, B_{k+1}\):
\begin{equation}\label{eq:update5}
\hat{B}_{k+1} = [ B_{k}(:,\hat{p}_{k}), B_{k+1}^{\star}], B_{k+1} = [\bar{B}_{k+1},\hat{B}_{k+1}].
\end{equation}
Next, we illustrate how to compute the corresponding \(a\) and \(B\) given the bundle \(P\in \mathbb{R}^{n \times l}\):
\begin{equation*}
\mathcal{A}(P\text{diag}(u)P^{T})=\sum_{i=1}^{l}u_{i}\mathcal{A}(p_{i}p_{i}^{T})=\text{Avec}^{T}[\text{svec}(p_{1}p_{1}^{T}),\cdots,\text{svec}(p_{l}p_{l}^{T})]u=\text{Avec}^{T}(\text{Pvec})u,
\end{equation*}
where \(\text{Avec}=[\text{svec}(A_{1}),\cdots,\text{svec}(A_{m})]\in \mathbb{R}^{\frac{n(n+1)}{2}\times m}\) and \(\text{Pvec}=[\text{svec}(p_{1}p_{1}^{T}),\cdots,\text{svec}(p_{l}p_{l}^{T})]\in \mathbb{R}^{\frac{n(n+1)}{2}\times l}\). In Algorithm \ref{alg:2}, the matrix \(A_{i:j,s:d}\) denotes a submatrix formed by the elements from the \(i\)-th to \(j\)-th rows and from the \(s\)-th to \(d\)-th columns of matrix \(A\); if \(i = j\) or \(s = d\), only \(i\) or \(s\) is displayed respectively.
\begin{algorithm}\label{alg3.1}
\caption{Generate \(\text{Pvec},a,B\)}\label{alg:2}
\begin{algorithmic}[1]
    \Require $P\in \mathbb{R}^{n\times l},\text{Avec},\text{Cvec}=\text{svec}(C)$ 
    \State \(k=0,\text{Pvec}=\mathbf{0}\)
    \For{\(i=1:n\)}
    \State \(\text{Pvec}_{(k+1):(k+i-1),1:l}=\sqrt{2}\text{broadcast}(P_{i,1:l})\odot P_{1:(i-1),1:l}\)
    \State \(\text{Pvec}_{(k+i),1:l}=P_{i,1:l}\odot P_{i,1:l}\)
    \State \(k=k+i\)
    \EndFor
    \State \Return\(a=\text{Pvec}^{T}\text{Cvec},B=\text{Avec}^{T}\text{Pvec}\)
\end{algorithmic}
\end{algorithm}

Through simple calculations, we can determine that the time complexity of Algorithm \ref{alg3.1} is \(\mathcal{O}(mn^{2}l)\). Therefore, limiting the number of newly generated constraints \(l\) each time can significantly improve the practical convergence efficiency of the algorithm.  
\subsection{Termination criterion}\label{subsec:Termination}
In this subsection, we introduce the relative convergence criteria that will be used later. These criteria are commonly used in the algorithms SDPNAL+\cite{yang2015sdpnal+}, SDPT3\cite{toh1999sdpt3}, MOSEK\cite{aps2019mosek}, and SpecBM\cite{ding2023revisiting}, which we will compare in the numerical experiment section. For problems (\ref{eq:Psdp}) and (\ref{eq:Dsdp}), the termination criterion derives from the KKT conditions (\ref{kkt}). Adopting the relative optimality conditions, we define the relative primal infeasibility as  
\begin{equation*}
\delta_{1}=\frac{\|\mathcal{A}(X)-b\|_{2}}{1+\|b\|_{2}}=\frac{\|Bu-b\|_{2}}{1+\|b\|_{2}},
\delta_{2}=\min\{\lambda_{\min}(X),0\}.
\end{equation*}  
For relative dual infeasibility, we have  
\begin{equation*}
\delta_{3}=\frac{\|\mathcal{A}^{T}y-C+S\|_{F}}{1+\|C\|_{F}},
\delta_{4}=\max\{-\lambda_{\min}(S),0\}.
\end{equation*}  
We express the relative dual gap as  
\begin{equation*}
\delta_{5}=\frac{|\langle C,X\rangle-b^{T}y|}{1+|\langle C,X\rangle|+|b^{T}y|}=\frac{|a^{T}u-b^{T}y|}{1+|a^{T}u|+|b^{T}y|}.
\end{equation*}  
By definition, in the iteration process, the primal variables satisfy \(X \succeq 0\) and the slack variables take the form \(S = C - \mathcal{A}^{T}y\), so \(\delta_{2} =\delta_{3} = 0\). In addition, for the bundle method, we can measure the relative approximation accuracy of the lower approximation model to the primal problem as  
\begin{equation*}
\delta_{6}=\frac{|F(y_{k})-F_{k}(z_{k+1})|}{1+|F(y_{k})|}.
\end{equation*}  
Given the precision, we terminate the algorithm under the condition that  
\begin{equation}\label{Termination}
\max\{\delta_{1},\delta_{4},\delta_{5},\delta_{6}\}\leq \epsilon.
\end{equation}
\subsection{Rank prediction}
If the rank \( r^{\star} \) of the optimal solution \( X^{\star} \) remains unknown, we can predict it in the early stages of iterations. By selecting a sufficiently large prior rank \( r > r^{\star} \), during iterations, we calculate the gradient information \( V_{k} \in \mathbb{R}^{n \times r} \), which consists of the eigenvectors corresponding to the smallest \( r \) eigenvalues of \( S_{k} = C - \mathcal{A}^{T}z_{k} \), and add \( V_{k} \) to \( P_{k} \). This ensures that information related to the optimal rank accumulates in the bundle \( P_{k} \). We then analyze the first \( r \) largest singular values of the bundle \( P_{k} \) and the smallest \( r \) eigenvalues of \( S_{k} \) to derive a predicted rank \( r_{k} \) for \( r^{\star} \). When the predicted rank stays unchanged for multiple consecutive iterations, this value can be taken as the rank \( r^{\star} \) of the optimal solution. Specifically, when analyzing the first \( r \) largest singular values of the bundle \( P_{k} \in \mathbb{R}^{n \times l_{k}} \), we first perform its singular value decomposition (SVD), which is expressed as 
\begin{equation*}
P_{k}=L_{k}\Sigma_{k}R_{k},
\end{equation*}  
with \(\Sigma_{k}\in\mathbb{R}^{n\times l_{k}}\). We let \(\sigma_{k}\) denote the diagonal elements of \(\Sigma_{k}\), labeled \(\sigma_{k}^{i}\) (\(i = 1, \cdots, l_{k}\)) and satisfying  
\begin{equation*}
\sigma_{k}^{1}\geq\cdots\geq\sigma_{k}^{r^{\star}}>\sigma_{k}^{r^{\star}+1} \geq\cdots\geq\sigma_{k}^{r}\geq \cdots\geq \sigma_{k}^{l_{k}}\geq0.
\end{equation*}  
The above equation assumes \(\sigma_{k}^{r^{\star}} > \sigma_{k}^{r^{\star}+1}\). We take the difference of the above singular value sequence, considering the first \(r\) results as  
\begin{equation*}
\alpha_{k}^{i}=\sigma_{k}^{i}-\sigma_{k}^{i+1},i=1,\cdots,r.
\end{equation*}  
We assume that when sufficiently close to the optimal solution, a significant gap arises between the first \(r^{\star}\) singular values and the subsequent ones, i.e., \(\alpha_{k}^{r^{\star}} > \alpha_{k}^{i}\) for \(i \neq r^{\star}\), \(i = 1, \cdots, r\). Thus, we can predict the rank of the optimal solution by detecting the position of the largest gap, with \(r_{k} = \bar{r}_{k}\). Similarly, we analyze the distribution of the smallest \(r\) eigenvalues of \(C-\mathcal{A}^{T}z_{k}\) as  
\begin{equation*}
\lambda_{k}^{1}\geq\cdots\geq\lambda_{k}^{r^{\star}}>\lambda_{k}^{r^{\star}+1} \geq\cdots\geq\lambda_{k}^{r}\geq \cdots\geq \lambda_{k}^{l_{k}}\geq0.
\end{equation*}

It is also assumed in the above equation that \(\lambda_{k}^{r^{\star}} > \lambda_{k}^{r^{\star}+1}\). After sufficiently approaching the optimal solution, there is a significant gap between the first \(r^{\star}\) eigenvalues and the subsequent ones, so the rank of the optimal solution can be predicted by detecting this gap, where \(r_{k} = \hat{r}_{k}\). Although the two methods are similar, their combined use may be more effective in some cases, i.e., taking \(r_{k} = \max\{\hat{r}_{k}, \bar{r}_{k}\}\).

After determining the rank of the optimal solution, we can choose to retain only the most recently updated bundle information \(a_{k}^{\star}, B_{k}^{\star}, V_{k}\), and appropriately reduce the step size to improve numerical stability.

In Algorithm \ref{alg:4}, \( P_k \) represents the bundle at the \( k \)-th iteration, and \( D_k \) is a diagonal matrix composed of the smallest \( l_k \) eigenvalues of \(C-\mathcal{A}^T z_k\) at the \( k \)-th iteration. Algorithm \ref{alg:4} is a single-iteration process. In actual invocation, it needs to be called continuously for multiple times until the condition \(\text{predcount}\leq \text{predcountmax}\) is violated, after which the rank \(r^{\star}\) of the predicted optimal solution is output.  
\begin{algorithm}
\caption{Single Iteration Rank Prediction}\label{alg:4}
\begin{algorithmic}[1]
    \Require $\text{priori rank } r > r^{\star},P_{0},D_{0},\text{predcount}=0,\text{predcountmax},r_{0}=0 $
    \If{\(\text{predcount}\leq \text{predcountmax}\)}
        \State \(\sigma_{k}=\text{SVD}(P_{k})\) \text{and} \(\sigma_{k}\in R^{l_{k}}\);
        \State \(\lambda_{k}=\text{diag}(D_{k})\) \text{and} \(\lambda_{k}\in R^{l_{k}}\);
        \State \(\bar{\sigma}_{k}=\sigma_{k}(1:r)-\sigma_{k}(2:r+1)\)
        \State \(\hat{\lambda}_{k}=\lambda_{k}(1:r)-\lambda_{k}(2:r+1)\)
        \State Find the position \( \bar{r}_{k} \) where \( \bar{\sigma}_{k} \) reaches its maximum value.
        \State Find the position \( \hat{r}_{k} \) where \( \hat{\lambda}_{k} \) reaches its maximum value.
        \State \(r_{k}=\max\{\bar{r}_{k},\hat{r}_{k}\}\)
        \If {\(r_{k}=r_{k-1}\)}
            \State \text{predcount}=\text{predcount}+1;
        \Else
            \State \text{predcount}=0;
        \EndIf
    \Else
        \State \(t=0.5t,y_{k}=z_{k}\)
        \State \(\bar{a}_{k}=0,\bar{B}_{k}=0,W_{k}=0\)
        \State \(a_{k}=[\bar{a}_{k};a_{k}^{\star}],B_{k}=[\bar{B}_{k},B_{k}^{\star}],P_{k}=V_{k}\)
        \State \(r^{\star}=r_{k}\)
    \EndIf
    \State \Return \(r^{\star}\)
\end{algorithmic}
\end{algorithm}

The parameter \(\text{predcountmax}\) refers to the number of consecutive times the predicted ranks are the same. A too small predcountmax may lead to a significant increase in the probability of incorrect predicted ranks, while a too large predcountmax may reduce the solution efficiency. Therefore, we set predcountmax=10 here. Employing Algorithm \ref{alg:4} for rank prediction allows the rank of the optimal solution to the problem to be ascertained with negligible loss of computational efficiency, while exhibiting a certain level of robustness. Detailed results of the numerical experiment are presented in Section \ref{subsubsec5.1.2} 
\subsection{Polyhedral bundle algorithm for solving SDP}
After the \(k\)-th iteration, the new candidate point \(z_{k+1}\) satisfies:  
\begin{equation*}
z_{k+1} = \underset{z \in \mathbb{R}^{m}}{\arg\min} \, F_{k}(z) + \frac{1}{2t}\|z - y_{k}\|_{2}^{2}.
\end{equation*}  
It follows that \(F_{k}(z_{k+1}) + \frac{1}{2t}\|z_{k+1} - y_{k}\|_{2}^{2} \leq F_{k}(y_{k}) \leq F(y_{k})\). Thus, when \(z_{k+1} \neq y_{k}\), we define \(\Delta_{k}^{\text{pred}} = F(y_{k}) - F_{k}(z_{k+1}) \geq \frac{1}{2t}\|z_{k+1} - y_{k}\|_{2}^{2} > 0\) as the predicted decrease, and \(\Delta_{k}^{\text{true}} = F(y_{k}) - F(z_{k+1})\) as the true decrease. The ratio of these two decreases is used to decide whether to accept the candidate point \(z_{k+1}\) and adjust the step size \(t\).  

Given parameters \(0 < \beta_{3} < \beta_{1} < \beta_{2} < 1\), if \(\beta_{1}\Delta_{k}^{\text{pred}} \leq \Delta_{k}^{\text{true}}\), the decrease in the objective function at \(z_{k+1}\) is acceptable, so the candidate point is accepted. Further, if \(\beta_{2}\Delta_{k}^{\text{pred}} \leq \Delta_{k}^{\text{true}}\), the decrease is sufficiently large, and the step size can be increased to accelerate convergence. Conversely, if \(\beta_{1}\Delta_{k}^{\text{pred}} > \Delta_{k}^{\text{true}}\), the candidate point is rejected. Meanwhile, if \(\beta_{3}\Delta_{k}^{\text{pred}} \geq \Delta_{k}^{\text{true}}\) occurs consecutively for multiple times, the step size needs to be appropriately reduced. We summarize the content in Section \ref{sec:3} as Algorithm \ref{alg:3}.

\begin{algorithm}
\caption{Polyhedral Bundle Method for solving SDP}\label{alg:3}
\begin{algorithmic}[1]
    \Require $ F(y);\text{ initial step size } t_{0}>0,0<t_{\min}<t_{\max}<\infty;0 < \beta_{3} < \beta_{1} < \beta_{2} < 1;\text{maxiter}> 0;\text{nullcount}=0,\text{nullmax}\geq0,\text{rank } r^{\star}>0,a_{0},B_{0},V_{0},W_{0},\epsilon>0,l_{\max}$
    \For{\(k=1:\text{maxiter}\)}
        \If{\text{Rank prediction is needed}}
            \State \text{Run Algorithm \ref{alg:4} and update} \(r^{\star},l_{\max}\) 
        \EndIf
        \State \text{Solving subproblem (\ref{eq:sub5}) yields \(u_{k}\) and the candidate point \(z_{k+1}\).}
        \State \text{Compute the \(r^{\star}\) top eigenvectors \(V_{k+1}\) of \(\mathcal{A}^{T}z_{k+1}-C\)}.
        \State \text{Compute }\(a_{k+1}^{\star},B_{k+1}^{\star}\) \text{by} Alogrithm \ref{alg:2}
        \State Update \( W_{k+1},P_{k+1},a_{k+1},B_{k+1} \) by (\ref{eq:update1})-(\ref{eq:update5})
        \State \text{Compute} \(\Delta_{k}^{\text{pred}},\Delta_{k}^{\text{true}}\)
        \If{\(\beta_{1}\Delta_{k}^{\text{pred}} \leq \Delta_{k}^{\text{true}}\)}
            \State Set \(y_{k+1}=z_{k+1}\) \text{ and } \text{nullcount}=0
            \If{\(\beta_{2}\Delta_{k}^{\text{pred}} \leq \Delta_{k}^{\text{true}}\)}
                \State \(t=\min\{2t,t_{\max}\}\)
            \EndIf
        \Else
            \State Set \(y_{k+1}=y_{k}\) \text{ and } \text{nullcount}=\text{nullcount}+1
            \If{\(\beta_{3}\Delta_{k}^{\text{pred}}\geq\Delta_{k}^{\text{true}}\text{ and nullcount}\geq\text{nullmax}\)}
                \State \(t=\max\{0.5t,t_{\min}\}\)
                \State \text{nullcount}=0
            \EndIf
        \EndIf
        \If{termination criterion is satisfied}
            \State \Return The $\epsilon$-optimal solution $(X,y,S)$.
            \State Quit.
        \EndIf
    \EndFor
\end{algorithmic}
\end{algorithm}

\section{Theoretical analysis}\label{sec:4} 
In this section, we first prove that under the condition of fixed step size \(t\), the update process of the lower approximation model satisfies the following lemma \ref{lemma:2}. Then, by combining the convergence of existing bundle methods, we present the convergence of the polyhedral bundle method.
\begin{lemma}\label{lemma:2}
The update process of the lower approximation model \(F_k(y)\) via (\ref{eq:update1})–(\ref{eq:update5}) satisfies conditions (\ref{eq:cond1})–(\ref{eq:cond3}).
\end{lemma}
\begin{proof}
First, we have:
\begin{equation*}
\begin{aligned}
F_{k+1}(y)&=-b^{T}y+\underset{\eta \geq 0, x \geq 0, \eta + \mathbf{1}^{T}x \leq \rho}{\max}\langle \eta W_{k+1} + P_{k+1}\text{diag}(x)P_{k+1}^{T}, \mathcal{A}^{T}y-C\rangle\\
&\leq-b^{T}y+\underset{X\in \mathbb{S}_{+}^{n}}{\max}\langle X,\mathcal{A}^{T}y-C\rangle=F(y).
\end{aligned}
\end{equation*}
Therefore, condition (\ref{eq:cond1}) is evident. In the above formula, it is only required that the column vectors of \(P_{k+1}\) are unit-normalized, and orthogonality is not necessary. 

Second, suppose that \(v\) is the eigenvector corresponding to the smallest eigenvalue of \(C-\mathcal{A}^{T}z_{k+1}\). According to the construction of \(P_{k+1}\), we know there exists a basis vector 
\begin{equation*}
s = (0,\cdots,0,1,0,\cdots,0)^{T} \text{ such that } P_{k+1}s = v.
\end{equation*}

Thus, \(\rho vv^{T} = \rho P_{k+1}ss^{T}P_{k+1}^{T} = P_{k+1}\text{diag}(\rho s)P_{k+1}^{T}\), and \(\mathbf{1}^{T}\rho s = \rho\mathbf{1}^{T}s = \rho\). Therefore, when \(\eta = 0\), we have \(\rho vv^{T} \in \{\eta W_{k+1} + P_{k+1}\text{diag}(x)P_{k+1}^{T} \mid \eta \geq 0, x \geq 0, \eta + \mathbf{1}^{T}x \leq \rho\}\). Thus, when \(\lambda_{\max}(\mathcal{A}^{T}z_{k+1} - C) > 0\), we have:
\begin{equation*}
\begin{aligned}
F_{k+1}(y)&=-b^{T}y+\underset{\eta \geq 0, x \geq 0, \eta + \mathbf{1}^{T}x \leq \rho}{\max}\langle \eta W_{k+1} + P_{k+1}\text{diag}(x)P_{k+1}^{T}, \mathcal{A}^{T}y-C\rangle\\ 
&\geq -b^{T}y+\langle \rho vv^{T}, \mathcal{A}^{T}y-C\rangle\\ 
&=-b^{T}z_{k+1}+\langle \rho vv^{T}, \mathcal{A}^{T}z_{k+1}-C\rangle+\langle -b+\rho \mathcal{A}(vv^{T}),y-z_{k+1}\rangle\\ 
&=F(z_{k+1})+\langle g_{k+1},y-z_{k+1}\rangle,
\end{aligned}
\end{equation*}
where \(g_{k+1}=-b+\rho \mathcal{A}(vv^{T})\in \partial F(z_{k+1})\). On the other hand, if \(\lambda_{\max}(\mathcal{A}^{T}z_{k+1} - C) \leq 0\), since \(0 \in W_{k+1}\), it also follows that
\begin{equation*}
F_{\mathcal{W}_{4,k+1}}(y)=-b^{T}y=-b^{T}z_{k+1}+\langle -b,y-z_{k+1}\rangle =F(z_{k+1})+\langle g_{k+1},y-z_{k+1}\rangle,
\end{equation*}
where \(g_{k+1} = -b \in \partial F(z_{k+1})\). Therefore, condition (\ref{eq:cond2}) holds.

Finally, we prove that condition (\ref{eq:cond3}) is satisfied. Assume that \(u_{k}^{\star}=[\eta_{k}^{\star};x_{k}^{\star}]\) is the optimal solution to Problem (\ref{eq:sub4}), then
\begin{equation*}
\begin{aligned}
W_{k}^{\star} &= \eta_{k}^{\star}W_{k} + P_{k}\text{diag}(x_{k}^{\star})P_{k}^{T} \\ 
&= \eta_{k}^{\star}W_{k} + P_{k}(:,\bar{p}_{k})\text{diag}(x_{k}^{\star}(\bar{p}_{k}))P_{k}^{T}(:,\bar{p}_{k}) + P_{k}(:,\hat{p}_{k})\text{diag}(x_{k}^{\star}(\hat{p}_{k}))P_{k}^{T}(:,\hat{p}_{k}) \\
&= \eta_{k}^{\star}W_{k} + \bar{P}_{k}\text{diag}(x_{k}^{\star}(\bar{p}_{k}))\bar{P}^{T} + [\hat{P}_{k},V_{k+1}]\text{diag}([x_{k}^{\star}(\hat{p}_{k});\textbf{0}])[\hat{P}_{k},V_{k+1}]^{T} \\ 
&= (\eta_{k}^{\star} + \textbf{1}^{T}x_{k}^{\star}(\bar{p}_{k}))\frac{\eta_{k}^{\star}W_{k} + \bar{P}_{k}\text{diag}(x_{k}^{\star}(\bar{p}_{k}))\bar{P}^{T}}{\eta_{k}^{\star} + \textbf{1}^{T}x_{k}^{\star}(\bar{p}_{k})} + P_{k+1}\text{diag}([x_{k}^{\star}(\hat{p}_{k});\textbf{0}])P_{k+1}^{T} \\ 
&= (\eta_{k}^{\star} + \textbf{1}^{T}x_{k}^{\star}(\bar{p}_{k}))W_{k+1} + P_{k+1}\text{diag}([x_{k}^{\star}(\hat{p}_{k});\textbf{0}])P_{k+1}^{T},
\end{aligned}
\end{equation*}
where \(\text{Tr}(W_{k+1})=\text{Tr}(\frac{\eta_{k}^{\star}W_{k} + \bar{P}_{k}\text{diag}(x_{k}^{\star}(\bar{p}_{k}))\bar{P}^{T}}{\eta_{k}^{\star} + \textbf{1}^{T}x_{k}^{\star}(\bar{p}_{k})}) \leq 1\), \(\eta_{k}^{\star} + \textbf{1}^{T}x_{k}^{\star}(\bar{p}_{k}) \geq 0\), \([x_{k}^{\star}(\hat{p}_{k});\textbf{0}] \geq 0\), and \(\eta_{k}^{\star} + \textbf{1}^{T}x_{k}^{\star}(\bar{p}_{k}) + \textbf{1}^{T}[x_{k}^{\star}(\hat{p}_{k});\textbf{0}] = \eta_{k}^{\star} + \textbf{1}^{T}x_{k}^{\star} \leq \rho\). Therefore,
\begin{equation*}
W_{k}^{\star} = (\eta_{k}^{\star} + \textbf{1}^{T}x_{k}^{\star}(\bar{p}_{k}))W_{k+1} + P_{k+1}\text{diag}([x_{k}^{\star}(\hat{p}_{k});\textbf{0}])P_{k+1}^{T} \in \mathcal{W}_{4,k+1}.
\end{equation*}
Therefore 
\begin{equation*}
\begin{aligned}
F_{k+1}(y)&=-b^{T}y+\underset{W\in \mathcal{W}_{4,k+1}}{\max}\langle W, \mathcal{A}^{T}y-C\rangle\\ 
&\geq -b^{T}y+\langle W_{k}^{\star}, \mathcal{A}^{T}y-C\rangle\\
&=-b^{T}z_{k+1}+\langle W_{k}^{\star}, \mathcal{A}^{T}z_{k+1}-C\rangle+\langle -b+\mathcal{A}W_{k}^{\star}, y-z_{k+1}\rangle\\
&=F_{k}(z_{k+1})+\langle s_{k+1}, y-z_{k+1}\rangle,
\end{aligned}
\end{equation*}
where \(s_{k+1}=-b+\mathcal{A}W_{k}^{\star}=\frac{y_{k}-z_{k+1}}{t_{k}}\in \partial F_{k}(z_{k+1})\).
\hfill\fbox{\rule{0pt}{0.42mm}\rule{0.42mm}{0pt}}
\end{proof}

The convergence of Algorithm \ref{alg:1} has been discussed in \cite{diaz2023optimal,monteiro2024parameter}, i.e., the following Theorem \ref{theorem:2}.
\begin{theorem}\label{theorem:2}
Consider a convex and \(M\)-Lipschitz function \(F(y)\) in
(\ref{eq:F}). Let \(F^{\star}=\underset{y\in R^{m}}{\inf} F(y) \) and \(\mathcal{C} = \{y \in R^{m} | F(y) = F^{\star}\}\). If \(\mathcal{C}\) is nonempty, the number of steps for Algorithm \ref{alg:1} before reaching an \(\epsilon> 0\) optimality, i.e. \(F(y)-F^{\star} \leq \epsilon\), is bounded by
\begin{equation*}
k\leq \mathcal{O}(\frac{12M^{2}D^{4}}{\beta(1-\beta)t\epsilon^{3}}),
\end{equation*}
where \(D = \underset{k}{\sup} \text{dist}(y_{k},\mathcal{C})<\infty\).
\end{theorem}
\begin{lemma}\label{lemma4.2}
The penalty function (\ref{eq:penalty}) is \((\|b\|_{2}+\rho \mathcal{A}_{\text{op}})\)-Lipschitz continuous.
\end{lemma}
\begin{proof}
    For any \(x, y \in \mathbb{R}^m\):
    \begin{equation*}
    \begin{aligned}
         \|F(x)-F(y)\|_{2}&=\|-b^{T}(x-y)+\rho(\max\{\lambda_{\max}(\mathcal{A}^{T}x-C),0\}-\max\{\lambda_{\max}(\mathcal{A}^{T}y-C),0\})\|_{2}\\
         &\leq \|-b^{T}(x-y)\|_{2}+\rho\|\max\{\lambda_{\max}(\mathcal{A}^{T}x-C),0\}-\max\{\lambda_{\max}(\mathcal{A}^{T}y-C),0\}\|_{2}\\
         &\leq \|b\|_{2}\|x-y\|_{2}+\rho\|\lambda_{\max}(\mathcal{A}^{T}x-C)-\lambda_{\max}(\mathcal{A}^{T}y-C)\|_{2}\\
         &\leq \|b\|_{2}\|x-y\|_{2}+\rho\|\mathcal{A}^{T}(x-y)\|_{\text{op}}\\
         &\leq \|b\|_{2}\|x-y\|_{2}+\rho\mathcal{A}_{\text{op}}\|x-y\|_{2}\\
         &=(\|b\|_{2}+\rho \mathcal{A}_{\text{op}})\|x-y\|_{2}.
    \end{aligned}
    \end{equation*}
    In the above derivation, \(\mathcal{A}_{\text{op}}=\mathcal{A}^{T}_{\text{op}}\) is the operator norm of \(\mathcal{A}^{T}\), i.e., \(\mathcal{A}^{T}_{\text{op}} = \max_{\|y\| = 1} \|\mathcal{A}^{T}y\|\). Therefore, the penalty problem (\ref{eq:penalty}) is \((\|b\|_2 + \rho \mathcal{A}_{\text{op}})\)-Lipschitz continuous.
    \hfill\fbox{\rule{0pt}{0.42mm}\rule{0.42mm}{0pt}}
\end{proof}

For our Algorithm \ref{alg:3}, assuming a fixed step size \(t\), with the rank of the optimal solution \(r^{\star}\) being known and the conditions of the aforementioned Lemma \ref{lemma:2} being satisfied, the above Theorem \ref{theorem:2} can be directly applied. The coefficients in \(\mathcal{O}\) depend only on the Lipschitz constant of Problem \ref{eq:penalty} (see Lemma \(\ref{lemma4.2}\)), the supremum \(D\) of the distances from the iteration sequence to the optimal solution set \(\mathcal{C}\), the parameter \(\beta_{1}\), and the lower bound of the step size \(t_{\min}\), therefore, the upper bound on the number of iterations required for Algorithm \(\ref{alg:3}\) to achieve \(\epsilon\)-optimality satisfies
\begin{equation*}
k\leq \mathcal{O}\left(\frac{1}{\epsilon^{3}}\right).
\end{equation*}
\section{Numerical experiment}\label{sec:5}
In this section, we test and compare our algorithm (denoted as OURS) with several benchmark solvers on random SDP problems and MaxCut problems, including MOSEK\cite{aps2019mosek}, SDPT3\cite{toh1999sdpt3}, SDPNAL+\cite{yang2015sdpnal+}, and \((r_{p}, r_{c})\)-SBMD\cite{liao2023overview}. Among them, MOSEK and SDPT3 are two SDP solvers based on second-order interior-point methods; SDPNAL+ is a SDP solver based on the semismooth Newton-conjugate gradient method; \((r_{p}, r_{c})\)-SBMD is an implementation of the spectral bundle method for the dual SDP (\ref{eq:Dsdp}). Given an initial rank \(r\), which is generally set as the rank of the optimal solution, it satisfies \(r = r_p + r_c\) during iterations. This means that after each iteration, the \(r_c\) bundles with the least information content are aggregated, leaving \(r_p\) bundles with the most information content. Then, the eigenvectors corresponding to the smallest \(r_c\) eigenvalues of the new dual variable are computed, i.e., \(r_c\) new bundles are added, so that the number of bundles in the next iteration remains \(r\).

In the above algorithms, relative optimality conditions are used to terminate the algorithms. For MOSEK, SDPT3, and SDPNAL+, four optimality conditions are adopted to terminate the algorithms: relative primal infeasibility, relative dual infeasibility, relative dual gap (it should be noted that the termination criterion of SDPNAL+ does not include the relative dual gap), and relative complementary slackness. For OURS, we will use (\ref{Termination}) in Section \ref{subsec:Termination} as the optimality criterion, and the stopping criterion of \((r_p, r_c)\)-SBMD is the same as that of OURS.

All the following numerical experiments were run on a laptop equipped with a 14-core Intel i9-13900H CPU @2.60GHz and 32GB RAM, using the MATLAB R2023a computing platform. The MATLAB code of OURS is available on \url{https://github.com/guran1214/Polyhedral-Bundle-Method-for-semidefinite-programming}. The letter \(\textbf{T}\) denotes that the given precision was not achieved within the specified solution time or number of iterations, while \(\textbf{OOM}\) indicates that the algorithm encountered an out-of-memory situation during iterations. In the following experiments, we set the precision to a moderate level of \(\epsilon=10^{-4}\), and we set the upper limit of the solution time to 1800 seconds. For the number of iterations and other parameters, Mosek, SDPT3, and SDPNAL+ all use their default settings; for OURS and \((r_p, r_c)\)-SBMD, we set the number of iterations to 500.

For OURS, we set the lower bound of the step size as \(t_{\min}=10^{-3}\) and the upper bound of the step size as \(t_{\max}=1\). The rank used in iterations is the rank of the given optimal solution, i.e., \(r=r^{\star}\), and the penalty parameter is \(\rho=2\text{Tr}(X^{\star})+1\). For the algorithm \((r_{p}, r_{c})\)-SBMD, we retain the basic parameter settings and set \(r_{p}=0\) and \(r_{c}=r^{\star}\). In subsequent experiments, we refer to \((r_p, r_c)\)-SBMD as SpecBM.

\subsection{Random SDP}
\subsubsection{Data generation}
Here, we make slight modifications to the random SDP problem generation algorithm in \cite{liao2023overview} and present the algorithm for generating random sparse standard SDP problems:
\begin{algorithm}
\caption{Algorithm for generating random sparse SDP problems with a primal variable of rank \( r \) that satisfies the strict complementarity slackness condition}\label{alg:5}
\begin{algorithmic}[1]
    \Require $n,m,r,\bar{\rho},s$ 
    \State \text{Randomly generate a sparse} \(X=s\bar{X}\) \text{such that} \(X\in \mathbb{S}^{n}_{++}\) \text{with sparsity } \( \bar{\rho}\)
    \State \text{Compute eigen-decomposition on} \(X\) as \(X = U_{1}\Sigma_{1}U_{1}^{T} + U_{2}\Sigma_{2}U_{2}^{T},U_{1}\in \mathbb{R}^{n\times r},U_{2}\in \mathbb{R}^{n\times (n-r)}\)
    \State\text{Randomly generate the sparse constraint matrices} \(A_{1},A_{2},A_{3}\cdots,A_{m} \in \mathbb{S}
^{n}\) \text{with sparsity } \( \bar{\rho}\), \text{where }\(\text{tr}(A_{i}) = 0,A_{i}=s\bar{A}_{i},i = 1, . . . , m\)
    \State \text{Randomly generate a dual solution} \(y^{\star}\in \mathbb{R}^{m}\)
    \State \text{Set} \(X^{\star} = U_{1}\Sigma_{1}U_{1}^{T},S^{\star}=U_{2}\Sigma_{2}U_{2}^{T},C=S^{\star}+\mathcal{A}^{T}y^{\star},b=\mathcal{A}(X^{\star})\)
    \State \Return \text{problem data} \((C,A_{1}, . . . , A_{m},b)\) \and a pair of optimal solutions \(X^{\star}, y^{\star}, S^{\star}\)
\end{algorithmic}
\end{algorithm}

In the random SDP generated by Algorithm \ref{alg:5}, all elements in \(\bar{X}, \bar{A}_{1}, \ldots, \bar{A}_{m}\) are drawn from the standard normal distribution \(N(0,1)\), with a sparsity of \(\bar{\rho}\), which denotes the proportion of non-zero elements in the matrices. The parameter \(s\) is used to adjust the magnitude of the condition number \(\bar{\kappa}(S^{\star})\) of the dual optimal solution \(S^{\star}\). In the following experiments, we use \(s = 1, 50\text{ and }500\) to generate SDP problems with low, medium, and high condition numbers respectively.
\subsubsection{Rank Prediction Test}\label{subsubsec5.1.2}
In this subsubsection, we introduce the results of the numerical experiments on the rank prediction algorithm. We select different priori ranks \( r \) and test the performance of rank prediction on different random SDP problems, where \( r_{0}=r^{\star} \) indicates that the rank of the optimal solution is known, so no prediction is performed. The theoretical studies by Barvinok\cite{barvinok1995problems} and Pataki \cite{burer2003nonlinear} indicate that the primal SDP (\ref{eq:Psdp}) always has a solution with rank \( r \), where \( \frac{r(r + 1)}{2} \leq m \); thus, we set \( r_{5}=\lceil \sqrt{2m} \rceil + 1 \).
\begin{table}
    \centering
    \caption{Rank prediction on random low-rank and low condition number SDP problems}\label{t2}
    \setlength\tabcolsep{0.9pt}
    \begin{tabular*}{\textwidth}{@{\extracolsep\fill}ccccccccc}
\toprule 
~&~&~& \multicolumn{6}{c}{CPU Time(s)(iterations)}  \\ 
\cmidrule{4-9}
n&m&\(r^{\star}\)&\(r_0=r^{\star}\)&\(r_1=r^{\star}+5\)&\(r_2=r^{\star}+10\)&\(r_3=r^{\star}+15\)&\(r_4=r^{\star}+20\)&\( r_5 = \lceil \sqrt{2m} \rceil + 1 \)\\
\midrule    
300&300&5&\textbf{0.57}(56) &0.78(68) &\textbf{0.57}(49)&0.99(81)&0.73(56)&0.71(61) \\
500&500&10&\textbf{1.14}(46) &1.45(53) &1.50(53) &1.52(52) &1.61(52)&\textbf{1.14}(41) \\
800&800&10&4.77(80) &\textbf{3.37}(57) &3.58(54) &3.59(54) &3.97(56)&3.74(52) \\
1000&1000&15&\textbf{10.61}(106) &17.00(153) &17.24(150) &36.06(328) &16.70(145)&10.65(103) \\
1500&1500&15&11.58(48) &19.27(76) &21.97(79) &\textbf{11.40}(40) &25.29(92)&29.72(106) \\
1000&10000&20&\textbf{14.79}(93) &19.32(108) &20.48(112) &17.72(100) &23.61(130)&25.62(115) \\
2000&2000&10&\textbf{18.53}(49) &25.48(64) &23.11(55) &21.92(48) &28.64(62)&38.99(69) \\
\hline
\end{tabular*}
\end{table}

For each test case in Table \ref{t2} that requires predicting the rank of the optimal solution (e.g., \( n=300 \), \( m=300 \), \( r^{\star}=5 \), \( r_1=10 \)), its solution process can be divided into two phases: In the first phase, the rank used is \( r_1=10 \), meaning 10 bundles are added in each iteration, and the upper limit of the bundle size at this stage is 65. After the rank of the optimal solution is predicted, the process moves to the second phase, where the predicted rank of the optimal solution (assuming the prediction is correct), \( r^{\star}=5 \), is adopted, with the bundle upper limit set to 20. When switching phases, excess bundles need to be aggregated to avoid exceeding the upper limit of the bundle size in the second phase. During the solution process of the quadratic subproblem, a larger value of the rank used leads to a greater number of bundles, which increases both the subproblem generation time and the subproblem solution time. However, since more bundles contain more information, this may reduce the total number of iterations and, conversely, might even decrease the total solution time. Therefore, whether to adopt the prediction algorithm and the size of the prior rank used are both issues that require trade-offs.  

Take the test case with \( n=800 \), \( m=800 \), \( r^{\star}=10 \) as an example: using the prediction method allows obtaining the problem’s solution faster. This is because the appropriate increase in the number of bundles in the first phase, while increasing the subproblem generation and solution time, does not cause excessive impact. Instead, the additional information from more bundles reduces the total number of iterations, thereby lowering the total solution time. For larger-scale problems, however (e.g., \( n=2000 \), \( m=2000 \), \( r^{\star}=10 \)), the impact of increased subproblem generation and solution time may offset the advantage of reduced total iterations. In addition, since the step size is dynamically adjusted during the algorithm’s iteration process, increasing the number of bundles may not necessarily reduce the number of iterations, which could in turn lead to an increase in the overall solution time.  

On the whole, assuming that the rank of the optimal solution is unknown, the increase in solution time (if such an increase exists) caused by using Algorithm \ref{alg:4} to make predictions for the problem is acceptable. In other words, this process possesses a certain degree of robustness. 
\subsubsection{Heuristic upper bound of bundles under low condition number}
For low-rank SDP problems with low condition numbers, we consider that they have relatively simple problem structures, and using fewer bundles enables the polyhedral bundle to approximate the penalty function well. To use as few bundles as possible to reduce memory consumption, we conduct experiments on some random SDP instances by setting different upper bounds of bundles. For these instances, we set the initial step size \(t_0 = 10^{-2}\) and parameters \(\beta_1 = 0.05\), \(\beta_2 = 0.65\), \(\beta_3 = 0.001\). Then, the upper bounds of bundles are set as \(\frac{1}{2}r^{\star}(r^{\star}+1)-r^{\star}\), \(\frac{1}{2}r^{\star}(r^{\star}+1)\), \(\frac{1}{2}r^{\star}(r^{\star}+1)+r^{\star}\), \((r^{\star})^2-r^{\star}\), \((r^{\star})^2\), \((r^{\star})^2+r^{\star}\), and \(\infty\) respectively. We observe whether these bundle upper bounds are sufficient to meet the needs of the lower approximation model for approximating the penalty function within the given solution time and number of iterations.
\begin{table}
    \centering
    \caption{Heuristic Bundle Upper Bound for Random low rank and low condition number Instances}
    \label{t3}
    \setlength\tabcolsep{0.9pt}
    \begin{tabular*}{\textwidth}{@{\extracolsep\fill}cccccccccc}
\toprule 
~&~&~& \multicolumn{7}{c}{CPU Time(s)(iterations)}  \\ 
\cmidrule{4-10}
n&m&\(r^{\star}\)&\(\frac{r^{\star}(r^{\star}+1)}{2}-r^{\star}\)&\(\frac{r^{\star}(r^{\star}+1)}{2}\)&\(\frac{r^{\star}(r^{\star}+1)}{2}+r^{\star}\)&\((r^{\star})^2-r^{\star}\)&\((r^{\star})^2\)&\((r^{\star})^2+r^{\star}\)&\(\infty\) \\
\midrule    
500&500&10&\textbf{T} &\textbf{T} &1.14(46) &1.27(51) &1.26(50)&1.36(53)&\textbf{1.07}(38) \\
500&1000&10&\textbf{T}&5.30(182)&\textbf{1.15}(44)&1.50(43)&1.42(44)&1.52(47)&1.30(42)\\
500&3000&10&\textbf{T} &\textbf{T} &2.35(75) &\textbf{2.31}(73) &2.45(72)&2.35(73)&2.89(72) \\
\midrule   
800&800&10&\textbf{T} &\textbf{T}  &4.77(80) &2.46(45) &2.53(43) &\textbf{2.36}(43) &2.66(45) \\
800&800&15&\textbf{T}&\textbf{T}&4.54(58)&4.35(53)&4.16(50)&\textbf{3.93}(48)&4.96(51)\\
800&800&25&\textbf{T} &\textbf{T}  &9.72(66) &\textbf{9.16}(52) &9.26(52) &9.30(52) &11.86(54) \\
\midrule  
1000&10000&20&\textbf{T} &\textbf{T} &\textbf{14.79}(89) &21.45(108) &19.50(97) & 16.70(85)&28.97(94) \\
1000&10000&25&\textbf{T}&\textbf{T} &\textbf{17.08}(82)&17.97(77)&18.58(78)&18.80(78)&22.53(79)\\
1000&10000&35&\textbf{T} &\textbf{T} &\textbf{22.73}(73) &26.53(67) &26.23(66) & 27.46(67)&26.92(66) \\
\midrule  
1500&1500&15&\textbf{T}&\textbf{10.47}(45)&11.58(48)&11.15(45)&11.11(46)&11.36(47)&11.35(45)\\
1500&1500&20&\textbf{T}&13.02(42)&11.34(41)&10.70(38)&\textbf{10.55}(37)&10.98(38)&11.33(38)\\
1500&1500&30&\textbf{T}&27.75(66)&23.96(59)&24.04(53)&24.18(53)&24.28(53)&25.69(54)\\
\midrule 
2000&2000&10&\textbf{T}&28.13(77)&18.53(49)&17.13(46)&\textbf{16.64}(45)&17.43(46)&17.65(47)\\
3000&3000&15&\textbf{T}&80.21(72)&72.96(66)&74.08(65)&\textbf{70.29}(63)&74.70(66)&73.51(65)\\
4000&4000&20&\textbf{T}&\textbf{T} &173.17(56) &162.58(54) &\textbf{156.55}(52) &160.89(53) &161.19(53)\\
\hline
\end{tabular*}
\end{table}

From the experimental data in Table \ref{t3}, we find that for low-rank SDP instances with low condition numbers, the impact of the number of bundles on problem solution efficiency is relatively stable and exhibits a certain degree of robustness. However, if the number of bundles is insufficient, it is often difficult to complete the solution within the given time or number of iterations. In addition, increasing the number of bundles can improve solution efficiency to a certain extent, but this improvement is not sustained because more bundles directly reduce the solution efficiency of subproblems, making the number of bundles a parameter that requires trade-off. Meanwhile, using fewer bundles can reduce memory consumption to a certain extent. Therefore, in the subsequent experiments on low condition number problems, we uniformly adopt the upper bound of the number of bundles as \(l_{\max}=\frac{1}{2}r^{\star}(r^{\star}+1)+r^{\star}\).  

We present the relationship between the relative optimality conditions and the number of iterations for the low-condition-number, low-rank SDP instance with \(n=500\), \(m=500\), and \(r^{\star}=10\) in Figure \ref{fig:1}. The red curve corresponds to the spectral algorithm SpecBM, and the blue curve represents our algorithm OURS. It can be seen from the figure that OURS has a faster and more stable convergence rate. From the perspective of relative primal optimality conditions and relative dual optimality conditions, the convergence rate of OURS is approximately linear.
\begin{figure}
 \centering
 \subfigure[Relative primal suboptimality]{\label{fig:subfig:a1} 
 \includegraphics[width=2.5in]{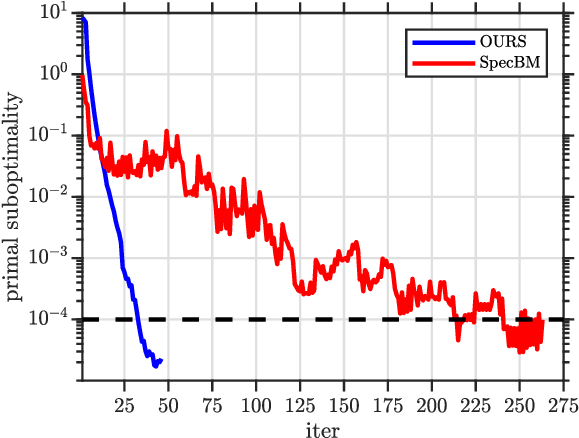}}
 \hspace{0.01in}
 \subfigure[Relative dual suboptimality]{\label{fig:subfig:b1} 
 \includegraphics[width=2.5in]{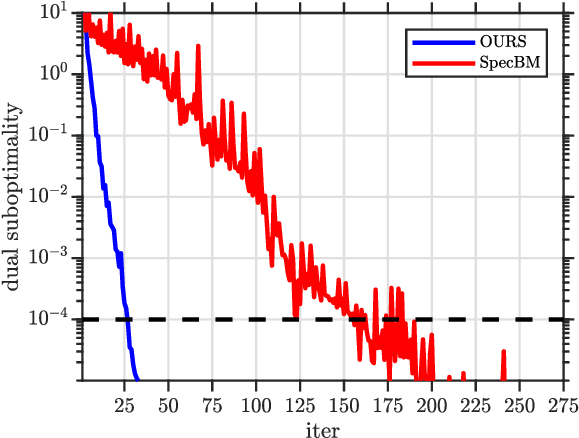}}
 \hspace{0.01in}
 \vfill
 \subfigure[Relative primal dual gap]{\label{fig:subfig:c1} 
 \includegraphics[width=2.5in]{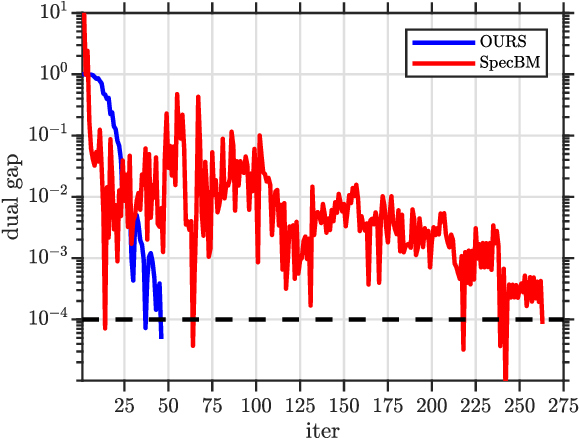}}
 \hspace{0.01in}
 \subfigure[Relative Accuracy]{\label{fig:subfig:d1} 
 \includegraphics[width=2.5in]{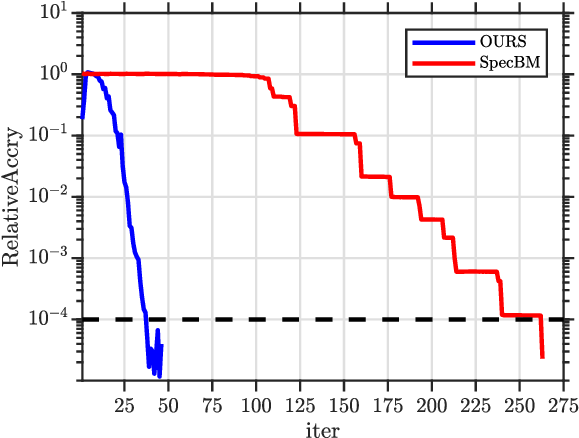}}
 \hspace{0.01in}\vskip -1mm
 \caption{The suboptimality of random instances with \(n=500, m=500\), and \(r^{\star}=10\).}\label{fig:1} 
\end{figure}

As shown in Table \ref{t4} below, we generated some small to medium-sized low-condition-number and low-rank SDP problems via Algorithm \ref{alg:5}, and controlled the sparsity of the problem data by adjusting the value of parameter \(\bar{\rho}\). For these problems, we set \(t_0 = 10^{-2}\), \(\beta_1 = 0.05\), \(\beta_2 = 0.65\), \(\beta_3 = 0.001\), and \(l_{\max}=\frac{1}{2}r^{\star}(r^{\star}+1)+r^{\star}\). Numerical experiments indicate that our algorithm can solve these instances more quickly: it significantly reduces the computation time compared with interior-point methods and SDPNAL+, and its performance is even more outstanding in comparison with SpecBM.
\begin{table}
    \centering
    \caption{Experimental results of Random sparse small and medium-sized instances}
    \label{t4}
    \setlength\tabcolsep{0.9pt}
    \begin{tabular*}{\textwidth}{@{\extracolsep\fill}ccccccccccc}
\toprule  
~&~&~&~&~&~& \multicolumn{5}{c}{CPU Time(s)(iterations)}  \\ 
\cmidrule{7-11}
\(\bar{\rho}\) & n &m& \(r^{\star}\) &\(\bar{\kappa}(X^{\star})\)&\(\bar{\kappa}(S^{\star})\)& OURS & SpecBM &SDPNAL+& MOSEK&SDPT3\\ 
\midrule    
    1e-2 & 300&300 &5 &2.42 &11.15 &\textbf{0.57}(56) &6.87(195)  &5.50 &0.66 &2.61 \\
    1e-2 & 300&300 &10&3.87 &11.15  &0.90(67) &20.91(220)  &2.78 &\textbf{0.68} &2.71 \\
    1e-2 & 300&300 &15&4.71 &11.15  &2.32(114) &46.47(236)  &3.28 &\textbf{0.63} &2.99 \\ 
    1e-2 & 300&300 &20&5.45 &10.96  &3.78(117) &114.04(238)  &4.20 &\textbf{0.64} &3.00 \\
\midrule    
    1e-3 & 500&500 &10&1.51 &8.15  &\textbf{1.14}(46) &26.01(262)  &9.20 &2.78&66.84\\
    1e-3 & 500&1000 &10&1.51 &8.15  &\textbf{1.15}(44) &26.99(204)  &7.91 &2.45&138.45\\
    1e-3 & 500&2000 &10&1.51 &8.15  &\textbf{2.17}(74) &46.66(239)  &10.65 &4.34&362.06\\ 
    1e-3 & 500&3000 &10&1.51 &8.15  &\textbf{2.35}(75) &57.06(218)  &11.67 &6.13&493.35\\ 
\midrule
    1e-3 & 800&800 &10&2.59 &8.58  &\textbf{4.77}(80) &71.14(263)  &20.22 &40.36 &119.35\\
    1e-3 & 800&800 &15&2.59 &8.58  &\textbf{4.54}(58) &122.96(237)  &16.63 &38.60 &115.12\\
    1e-3 & 800&800 &20&2.59 &8.58  &\textbf{7.51}(71) &242.15(265)  &52.13 &40.32 &116.02\\
    1e-3 & 800&800 &25&2.59 &8.58  &\textbf{9.72}(66) &522.44(292)  &16.76 &39.40 &116.95\\ 
\midrule    
    1e-4 & 1000&1000 &15&1.00 &8.46  &\textbf{10.61}(106) &111.01(275) &55.31 &9.38 &82.45\\
    1e-4 & 1000&2000 &20&1.01 &8.46  &\textbf{9.26}(74) &194.43(222) &55.99 &29.04 &346.52\\ 
    1e-4 & 1000&4000 &25&1.36 &8.46  &\textbf{10.57}(64) &411.81(212) &44.46 &71.04 &884.02\\ 
    1e-4 & 1000&6000 &30&1.36 &8.46  &\textbf{20.78}(83) &911.56(232) &75.45 &128.52 &\textbf{T}\\
\midrule    
    1e-4 & 1500&1500 &15&1.35 &8.15  &\textbf{11.58}(48) &255.69(280)  &142.02 &65.47&1058.5 \\
    1e-4 & 1500&1500 &20&2.07 &8.15  &\textbf{11.34}(41) &339.14(217)  &151.95 &63.77&988.57 \\ 
    1e-4 & 1500&1500 &30&2.07 &8.15  &\textbf{23.96}(59) &1111.5(228)  &117.08 &64.87&993.23 \\ 
    1e-4 & 1500&1500 &35&2.07 &8.15 &\textbf{25.49}(53) &\textbf{OOM}  &117.18 &84.08 &1134.2  \\ 
\midrule
    1e-4 & 1000&10000 &20&1.01 &8.46  &\textbf{14.39}(89) &\textbf{T}   &76.40  &45.05 &\textbf{T}\\ 
    1e-4 & 1000&10000 &25&1.36 &8.46  &\textbf{17.08}(82) &629.18(250)  &54.47 &43.15&\textbf{T}\\ 
    1e-4 & 1000&10000 &30&1.36 &8.46  &\textbf{27.07}(98) &983.51(224)  &72.21 &42.76&\textbf{T}\\
    1e-4 & 1000&10000 &35&1.36 &8.46  &\textbf{22.73}(73) &1267.3(220) &56.46 &45.44&\textbf{T}\\
\hline
\end{tabular*}
\end{table}

\begin{table}
    \centering
    \caption{Experimental results of Random sparse large-scale instances}
    \label{t5}
    \setlength\tabcolsep{0.9pt}
    \begin{tabular*}{\textwidth}{@{\extracolsep\fill}ccccccccccc}
\toprule  
~&~&~&~&~&~& \multicolumn{5}{c}{CPU Time(s)(iterations)}  \\ 
\cmidrule{7-11}
\(\bar{\rho}\) & n &m& \(r^{\star}\) &\(\bar{\kappa}(X^{\star})\)&\(\bar{\kappa}(S^{\star})\)& OURS & SpecBM &SDPNAL+& MOSEK&SDPT3\\ 
\midrule 
    1e-4 & 2000&2000 &10&4.30 &8.23 &\textbf{18.53}(49) &262.75(218) &250.11 &410.62&\textbf{T}\\ 
    1e-4 & 2000&4000 &10&4.30 &8.23 &\textbf{37.54}(98) &434.59(282)  &337.26 &1399.5&\textbf{T}\\
    1e-4 & 2000&6000 &10&4.30 &8.23 &\textbf{53.19}(130) &426.89(229) &360.55 &\textbf{T} &\textbf{T} \\
    1e-4 & 2000&8000 &10&4.30 &8.23 &\textbf{47.22}(108) &581.72(264) &410.54 &\textbf{T} &\textbf{T} \\
\midrule 
    1e-4 & 3000&3000 &15&2.07 &8.15  &\textbf{72.96}(66) &1082.9(216)  &819.91 &\textbf{T}&\textbf{T}\\  
    1e-4 & 3000&6000 &15&2.07 &8.15  &\textbf{99.99}(74) &1793.2(265)  &1048.6 &\textbf{T}&\textbf{T}\\ 
    1e-4 & 3000&9000 &15&2.07 &8.15  &\textbf{118.15}(81) &\textbf{T} &1360.0 &\textbf{T} &\textbf{T} \\
    1e-4 & 3000&12000 &15&2.07 &8.15  &\textbf{132.34}(86) &\textbf{T} &754.48 &\textbf{T} &\textbf{T} \\
\midrule 
    1e-4 & 4000&4000 &20&1.86 &8.15  &\textbf{173.17}(56) &\textbf{OOM}  &\textbf{T} &\textbf{T}&\textbf{T}\\
    1e-4 & 4000&8000 &20&1.86 &8.15  &\textbf{303.52}(86) &\textbf{OOM}  &\textbf{T} &\textbf{T}&\textbf{T}\\
    1e-4 & 4000&12000 &20&1.86 &8.15  &\textbf{392.41}(105) &\textbf{OOM} &\textbf{T} &\textbf{T} &\textbf{T} \\
    1e-4 & 4000&16000 &20&1.86 &8.15  &\textbf{449.07}(113) &\textbf{OOM} &\textbf{T} &\textbf{T} &\textbf{T} \\
\midrule 
    1e-4 & 5000&5000 &25&1.81 &8.59  &\textbf{586.89}(82) &\textbf{OOM} &\textbf{T} &\textbf{T} &\textbf{T} \\
    1e-4 & 5000&15000 &25&1.81 &8.59  &\textbf{742.57}(95) &\textbf{OOM} &\textbf{T} &\textbf{T} &\textbf{T} \\
    1e-4 & 5000&20000 &25&1.81 &8.59  &\textbf{803.91}(99) &\textbf{OOM} &\textbf{T} &\textbf{T} &\textbf{T} \\
    1e-4 & 5000&25000 &25&1.81 &8.59  &\textbf{1647.88}(190) &\textbf{OOM} &\textbf{T} &\textbf{T} &\textbf{T} \\
\midrule
    1e-4 & 7500&500 &15&1.36 &10.36  &\textbf{370.19}(62) &\textbf{OOM}  &\textbf{T} &\textbf{OOM}&\textbf{T}\\
    1e-4 & 10000&500 &15&1.70 &10.33  &\textbf{549.95}(50) &\textbf{OOM}  &\textbf{T} &\textbf{OOM}&\textbf{T}\\ 
    1e-4 & 12500&500 &15&2.02 &9.86  &\textbf{1363.58}(80) &\textbf{OOM}  &\textbf{T} &\textbf{OOM}&\textbf{T}\\
    1e-4 & 15000&500 &15&1.80 &10.36  &\textbf{1582.91}(58)&\textbf{OOM}&\textbf{T} &\textbf{OOM}&\textbf{T}\\
\hline
\end{tabular*}
\end{table}

As can be seen from Table \ref{t5}, in large-scale problems, OURs not only have an advantage in computation time but also can further reduce memory requirements. Based on the three groups of experiments mentioned above, OURs can be considered a promising implementation of the polyhedral bundle method.

\subsubsection{Medium and high condition numbers}
For all subsequent instances, the following parameters are set consistently across both medium- and high-condition cases: \(\beta_1 = 0.05\), \(\beta_3 = 0.001\), \(l_{\max}^{1} = \frac{1}{2}r^{\star}(r^{\star}+1)+r^{\star}\), \(l_{\max}^{2} = (r^{\star})^2\), and \(l_{\max}^{3} = 2(r^{\star})^2\). Only the parameters \(t_0\) and \(\beta_2\) differ between the two condition types. For medium-condition instances, we set \(t_0 = 10^{-5}\) and \(\beta_2 = 0.65\). For high-condition instances, we set \(t_0 = 10^{-7}\) and \(\beta_2 = 0.85\).
\begin{table}
    \centering
    \caption{Experimental results of median condition instances}
    \label{t6}
    \setlength\tabcolsep{0.9pt}
    \begin{tabular*}{\textwidth}{@{\extracolsep\fill}cccccccccc}
\toprule  
~&~&~&~&~&~& \multicolumn{4}{c}{CPU Time(s)(iterations)} \\
\cmidrule{7-10}
\(\bar{\rho}\) & n &m& \(r^{\star}\) &\(\bar{\kappa}(X^{\star})\)&\(\bar{\kappa}(S^{\star})\)& OURS(\(l_{\max}^{1}\)) & OURS(\(l_{\max}^{2}\))&OURS(\(l_{\max}^{3}\)) & SpecBM\\ 
\midrule 
    1e-3 & 500&500 &10&1.65 &358.83 &2.05(79) &\textbf{1.56}(57)&1.99(58) &\textbf{T}  \\
    1e-4 & 1500&1500 &20&2.45 &358.83 &\textbf{18.65}(65) &19.87(63)&22.80(62) &\textbf{T}  \\
    1e-4 & 2000&2000 &25&10.94 &362.79 &80.81(96) &\textbf{72.45}(86)&85.09(81) &\textbf{T}  \\
\hline
\end{tabular*}
\end{table}

\begin{table}
    \centering
    \caption{Experimental results of high condition instances}
    \label{t7}
    \setlength\tabcolsep{0.9pt}
    \begin{tabular*}{\textwidth}{@{\extracolsep\fill}cccccccccc}
\toprule  
~&~&~&~&~&~& \multicolumn{4}{c}{CPU Time(s)(iterations)} \\
\cmidrule{7-10}
\(\bar{\rho}\) & n &m& \(r^{\star}\) &\(\bar{\kappa}(X^{\star})\)&\(\bar{\kappa}(S^{\star})\)& OURS(\(l_{\max}^{1}\)) & OURS(\(l_{\max}^{2}\))& OURS(\(l_{\max}^{3}\)) & SpecBM\\ 
\midrule 
    1e-3 & 500&500 &10&1.65 &3579.4 &\textbf{T} &\textbf{2.28}(70)&3.31(67) &\textbf{T} \\
    1e-4 & 1500&1500 &20&2.46 &3579.4 &\textbf{T} &\textbf{27.72}(75)&39.06(75) &\textbf{T} \\
    1e-4 & 2000&2000 &25&11.32 &3619.0 &\textbf{T} &\textbf{72.83}(82)&124.65(80) &\textbf{T} \\
\hline
\end{tabular*}
\end{table}

As shown in Tables \ref{t6} and \ref{t7}, as the condition number of the problem increases, the problem structure becomes more complex, and the lower approximation model requires more bundles to enhance its approximation capability. However, excessively increasing the number of bundles may cause the generation process and solution process of the quadratic programming subproblem to take excessively long, resulting in an increase in the total solution time. Thus, it is necessary to appropriately increase the number of bundles.
\subsection{Max-Cut}
In this subsection, we conducted experiments on the combinatorial optimization problem Max-Cut, with data from \cite{davis2011university}, where \(L\) is the Laplacian matrix, and its positive semidefinite relaxation form is:
\begin{equation*}
\underset{X}{\min} \frac{1}{4}\langle L,X \rangle \text{ s.t. }X_{ii}=1,\forall i=1,\cdots,n\text{ and }X\succeq 0
\end{equation*}

Since the condition numbers of these problems are all relatively high, we thus set the parameters as \(t_0 = 0.01\), \(\beta_1 = 0.05\), \(\beta_2 = 0.65\), \(\beta_3 = 0.001\), \(l_{\max} = (r^{\star})^2\).
\begin{table}
    \centering
    \caption{Experimental results of Max-Cut instances}
    \label{t8}
    \setlength\tabcolsep{0.9pt}
    \begin{tabular*}{\textwidth}{@{\extracolsep\fill}ccccccccccc}
\toprule  
~&~&~&~&~&~& \multicolumn{5}{c}{CPU Time(s)(iterations)}  \\ 
\cmidrule{7-11}
 & n &m& \(r^{\star}\)&\(\bar{\kappa}(X^{\star})\) &\(\bar{\kappa}(S^{\star})\) & OURS & SpecBM &SDPNAL+& MOSEK&SDPT3 \\
\midrule    
    G1 & 800&800 &13&13.991&\(3.269\times 10^3\) &8.754(164) &33.18(158)  &17.08 &2.29&2.27\\
    G3 & 800&800 &14&187.37&\(1.590\times 10^3\) &11.62(202) &34.87(150)  &18.05 &2.34&2.94\\
    G5 & 800&800 &12&18.238&\(2.258\times 10^3\) &8.341(169) &23.32(130)  &22.02 &2.42&2.39\\
    G11 & 800&800 &6&10.137&\(3.399\times 10^5\) &\textbf{T} &\textbf{T}  &19.37 &2.15&1.58\\
    G12 & 800&800 &8&48.878&\(4.836\times 10^4\) &\textbf{T} &\textbf{T}  &20.61 &1.99&4.53\\
    G18 & 800&800 &10&13.422&\(6.181\times 10^3\) &9.175(202) &34.88(261)  &23.33 &2.12&2.41\\
    G19 & 800&800 &9&10.892&\(1.057\times 10^4\) &6.731(158) &\textbf{T}  &21.33 &2.40&2.27\\
    G21 & 800&800 &9&7.614&\(1.528\times 10^4\) &7.199(170) &\textbf{T}  &21.65 &2.14&2.11\\
    G23 & 2000&2000 &19&37.136&\(1.460\times 10^4\) &58.96(286) &\textbf{T}  &159.52 &30.47&21.92\\
    G24 & 2000&2000 &18&45.429&\(2.498\times 10^3\) &46.04(238) &291.54(169) &158.96 &28.86&20.19\\ 
    G35 & 2000&2000 &17&15.861&\(2.035\times 10^5\)&\textbf{T} &\textbf{T} &155.96 &30.24&21.34\\
    G43 & 1000&1000 &13&12.226&\(2.590\times 10^3\) &19.53(219) &54.73(195) &25.67 &3.91&3.67\\
    G44 & 1000&1000 &14 &100.50&\(1.222\times 10^3\)&12.27(198) &47.60(152)  &25.83 &4.04&3.67\\
    G45 & 1000&1000 &14 &16.478&\(1.650\times 10^3\)&17.02(229) &56.88(178)  &25.74 &3.88&3.79\\
    G46 & 1000&1000 &13 &8.345&\(1.353\times 10^6\)&16.71(207) &\textbf{T}  &25.96 &3.86&3.63\\
    G47 & 1000&1000 &14 &46.236&\(4.444\times 10^4\)&12.33(197) &\textbf{T}  &25.80 &3.96&3.79\\
\hline
\end{tabular*}
\end{table}

From the data in the table \ref{t8}, it can be observed that for problems with excessively large condition numbers (such as G11), it is difficult to meet the required precision within the given time or number of iterations. However, on the whole, our algorithm achieves higher solution efficiency on Max-Cut problems than SpecBM and SDPNAL+, but still cannot match the interior-point method.
\begin{figure}
 \centering
 \subfigure[Relative primal suboptimality]{\label{fig:subfig:a2} 
 \includegraphics[width=2.4in]{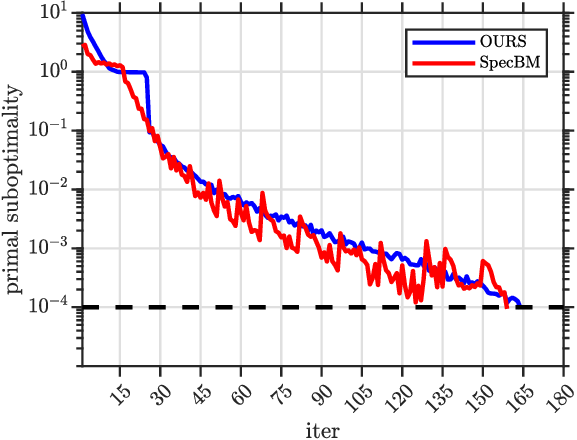}}
 \hspace{0.01in}
 \subfigure[Relative dual suboptimality]{\label{fig:subfig:b2} 
 \includegraphics[width=2.4in]{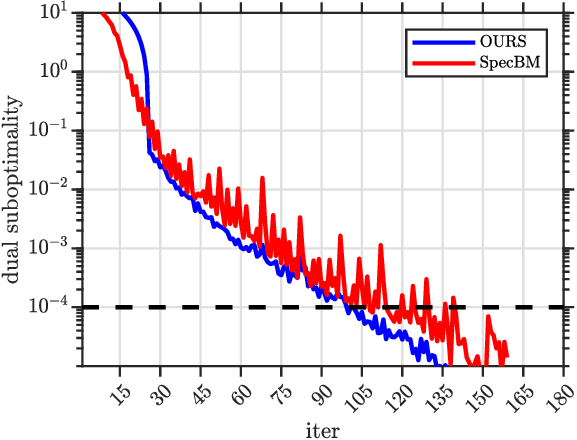}}
 \hspace{0.01in}
 \vfill
 \subfigure[Relative primal dual gap]{\label{fig:subfig:c2} 
 \includegraphics[width=2.4in]{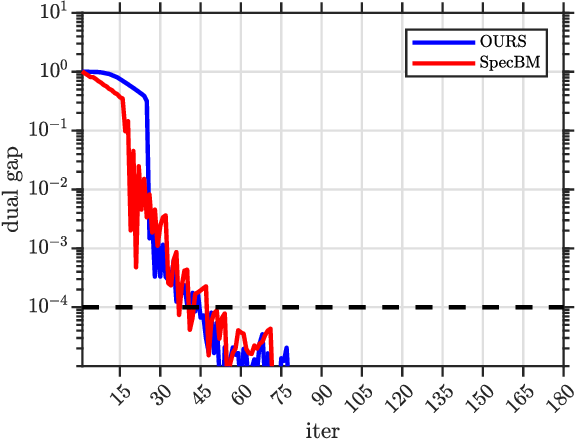}}
 \hspace{0.01in}
 \subfigure[Relative Accuracy]{\label{fig:subfig:d2} 
 \includegraphics[width=2.4in]{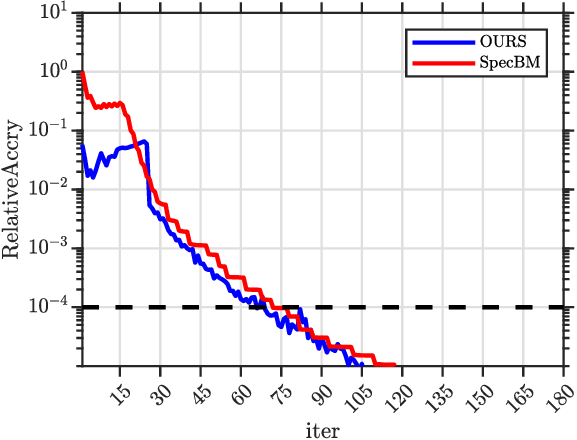}}
 \hspace{0.01in}\vskip -1mm
 \caption{The suboptimality of Max-Cut instances G1.}\label{fig:2} 
\end{figure}

By observing the relationship between the number of iterations and the optimality conditions of the two algorithms in Figure \ref{fig:2} above, we find that there is no significant overall difference between our algorithm and SpecBM, but the stability of our algorithm during the iteration process is better than that of SpecBM. Besides, we can also find from the relationship between CPU time and optimality conditions in Figure \ref{fig:3} below that the solution efficiency of our method is much higher than that of SpecBM.

\begin{figure}
 \centering
 \subfigure[Relative primal suboptimality]{\label{fig:subfig:a3} 
 \includegraphics[width=2.4in]{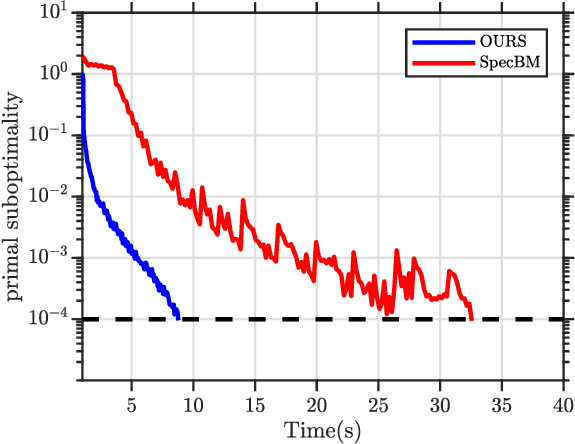}}
 \hspace{0.01in}
 \subfigure[Relative dual suboptimality]{\label{fig:subfig:b3} 
 \includegraphics[width=2.4in]{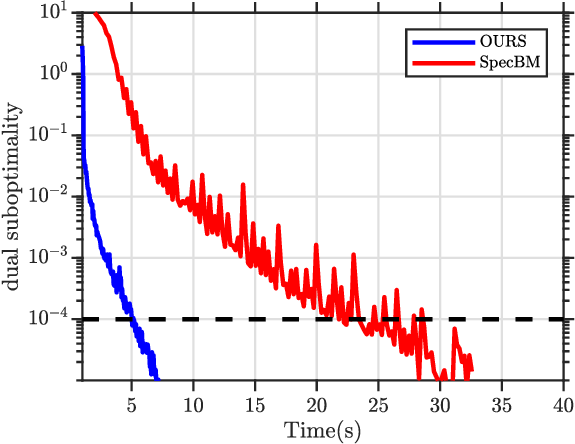}}
 \hspace{0.01in}
 \vfill
 \subfigure[Relative primal dual gap]{\label{fig:subfig:c3} 
 \includegraphics[width=2.4in]{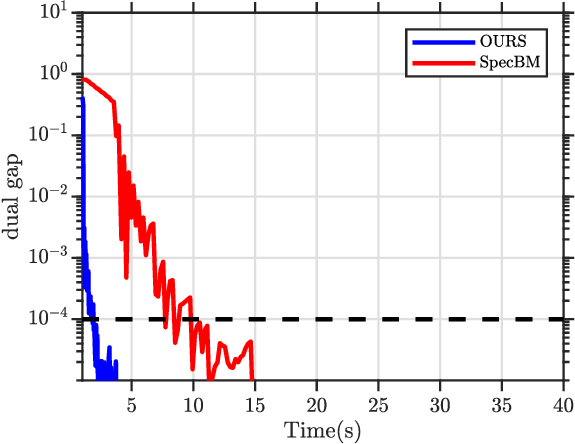}}
 \hspace{0.01in}
 \subfigure[Relative Accuracy]{\label{fig:subfig:d3} 
 \includegraphics[width=2.4in]{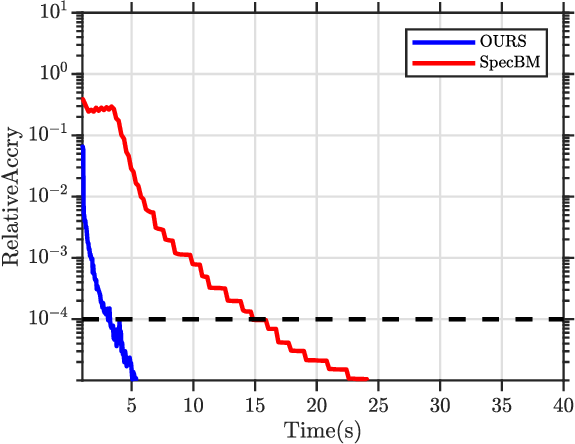}}
 \hspace{0.01in}\vskip -1mm
 \caption{The suboptimality of Max-Cut instances G1.}\label{fig:3} 
\end{figure}

For the three Max-Cut instances G11, G12, and G35, due to their excessively large condition numbers and complex problem structures, it is impossible to obtain solutions meeting the required precision within the given time or number of iterations. Thus, an intuitive idea is to increase the number of bundles so that the lower approximation model can better approximate such problems. First, we achieve this by increasing the upper bound of the bundle. Here, two new upper bounds for the bundle are set for experimentation, namely \(l_{\max}=2(r^{\star})^{2}\) and \(l_{\max}=5(r^{\star})^{2}\), with other parameters remaining unchanged.
\begin{table}
    \centering
    \caption{Convergence test of bundle upper bound on Max-Cut instances with high condition number}
    \label{t9}
    \setlength\tabcolsep{0.9pt}
    \begin{tabular*}{\textwidth}{@{\extracolsep\fill}ccccccccc}
\toprule  
~&~&~&~&~&~& \multicolumn{3}{c}{CPU Time(s)(iterations)} \\
\cmidrule{7-9}
 & n &m& \(r^{\star}\) &\(\bar{\kappa}(X^{\star})\)&\(\bar{\kappa}(S^{\star})\)& \((r^{\star})^{2}\) & \(2(r^{\star})^{2}\)&\(5(r^{\star})^{2}\)\\ 
\midrule 
    G11 & 800&800 &6&10.137 &\(3.399\times 10^5\)  &\textbf{T} &\textbf{T} &26.40(405)\\
    G12 & 800&800 &8&48.878 &\(4.836\times 10^4\)  &\textbf{T} &27.95(425) &56.53(392)  \\
    G35 & 2000&2000 &17&15.861 &\(2.035\times 10^5\) &\textbf{T} &175.40(365) &574.01(368)  \\
\hline
\end{tabular*}
\end{table}

The numerical experimental results in Table \ref{t9} show that increasing the upper bound of the number of bundles enables these three instances to meet the preset solution accuracy requirements within the given time and number of iterations. Additionally, increasing this upper bound can reduce the total number of iterations to some extent. However, it also raises the difficulty of solving the quadratic programming subproblems, which may potentially prolong the overall solution time. Therefore, once a given solution accuracy requirement is set, adjusting the upper bound of the number of bundles requires a trade-off.

Increasing the bundle upper bound is aimed at enhancing the approximation capability of the lower approximation model, that is, improving the amount of information that the model can capture. Therefore, in addition to changing the bundle upper bound, the approximation capability of the model can also be improved by adjusting the rank used in the iteration process. For example, for instance G11, the rank of its optimal solution is \( r^{\star}=6 \), and the rank used in the iteration process is also \( r=6 \), which means that after each iteration, the eigenvectors corresponding to the smallest \( r \) eigenvalues of the dual variable \( S \) are added to the bundle. At this time, the adopted bundle upper bound is \( l_{\max}=(r^{\star})^2=36 \). Thus, another approach is to increase the rank \( r \) used in the iteration process while setting the bundle upper bound as \( l_{\max}=r^2 \), which also serves to improve the information-capturing capability of the model. Here, we adopt the ranks \( r=r^{\star}+6 \) and \( r=r^{\star}+12 \).
\begin{table}
    \centering
    \caption{Convergence behavior of high condition number Max-Cut instances with different ranks adopted in the iteration}
    \label{t10}
    \setlength\tabcolsep{0.9pt}
    \begin{tabular*}{\textwidth}{@{\extracolsep\fill}ccccccccc}
\toprule  
~&~&~&~&~&~& \multicolumn{3}{c}{CPU Time(s)(iterations)} \\
\cmidrule{7-9}
 & n &m& \(r^{\star}\) &\(\bar{\kappa}(X^{\star})\)&\(\bar{\kappa}(S^{\star})\)& \(r^{\star}\) & \(r^{\star}+6\)&\(r^{\star}+12\)\\ 
\midrule 
    G11 & 800&800 &6&10.137 &\(3.399\times 10^5\)  &\textbf{T} &20.39(402)&38.29(424) \\
    G12 & 800&800 &8&48.878 &\(4.836\times 10^4\)  &\textbf{T} &12.61(215)&35.62(325) \\ 
    G35 & 2000&2000 &17&15.861 &\(2.035\times 10^5\) &\textbf{T} &160.11(442)&145.18(328) \\
\hline
\end{tabular*}
\end{table}

Table \ref{t10} shows that for Max-Cut problems with excessively large condition numbers, the polyhedral bundle method can also shorten the solution time by increasing the rank \(r\) used in the iteration process; similar to directly increasing the number of bundles, increasing the value of the rank can indirectly increase the number of bundles. However, excessively increasing the value of the rank may introduce excessive useless information and cause interference, leading to a decrease in the solution efficiency of subproblems; therefore, it is necessary to balance the increased value of the rank.

In summary, Tables \ref{t9} and \ref{t10} show that for Max-Cut problems with high condition numbers, the solution efficiency can be improved by appropriately increasing the upper bound of the number of bundles or the value of the rank used in iterations; however, how to appropriately adjust these two parameters remains an unsolved problem.
\section{Conclusion}\label{sec:6}
Based on the traditional polyhedral bundle method, this study proposes an improved subproblem, which can avoid the singularity of quadratic programming (QP) subproblems under general circumstances and reduce iterative constraints. By combining the convergence proof of the general bundle method, it is verified that the update process of this method satisfies its three basic conditions, thereby establishing the basic convergence. The upper limit of the number of bundles is heuristically selected using the rank of the optimal solution of the original problem, which ensures the solution efficiency of the QP subproblem. When the rank of the optimal solution is unknown, a rank prediction method can be adopted, and this method exhibits a certain degree of robustness.

This method performs well in handling random sparse SDP problems with low condition numbers. However, when applied to Max-Cut problems or problems with high condition numbers, an insufficient number of bundles will result in insufficient approximation capability of the lower approximation model and poor solution performance. In such cases, the overall solution efficiency can be improved by appropriately increasing the upper limit of the number of bundles or the rank used in iterations.

In summary, the polyhedral bundle method has advantages in dealing with SDP problems with low condition numbers, but the selection of the upper limit of the number of bundles is subjective and lacks theoretical analysis, and the impact of the condition number on the approximation capability of the lower approximation model is also unclear. Future research plans to study the quantitative relationship between parameters such as the rank of the optimal solution of the problem and the condition number, and the approximation capability of the lower approximation model, thereby proposing a better rank prediction method and an algorithm for adaptively adjusting the number of bundles.

\noindent {\bf Acknowledgments.} The authors thank Professor Ya-xiang Yuan for reading the manuscript of this work, which originated from the second author's Ph.D. research under his supervision. This work was supported by the National Key R\&D Program of China grant No.2022YFA1003800, the National Natural Science Foundation of China No. 12201318 and the Fundamental Research Funds for the Central Universities No.63253105.

\noindent {\bf Data Availability.} Data will be made available on reasonable request.
\bibliography{refs}
\end{document}